\documentclass[12pt]{article}
\usepackage{amsmath,amssymb,amsthm,latexsym,amsbsy,amsfonts,mathrsfs}
\usepackage{color,verbatim}
\usepackage{pstricks}
\usepackage{tabularx}
\usepackage{rotating}
\usepackage{graphicx,epsfig}
\usepackage{enumitem}
\usepackage[normalem]{ulem}
\usepackage{xcolor}
\newcommand{\C}{{\mathbb C}}
\newcommand{\R}{{\mathbb R}}

\newcommand{\N}{{\mathbb N}}
\newcommand{\mP}{\mathbb P}

\newcommand{\mB}{\mathbb B}
\newcommand{\mE}{\mathbb E}
\newcommand{\mF}{\mathbb F}
\newcommand{\mH}{\mathbb H}
\newcommand{\mL}{\mathbb L}

\newcommand{\mW}{\mathbb W}

\newcommand{\inpro}[2]{\left\langle{#1},{#2}\right\rangle}

\newcommand{\vecB}{\boldsymbol{B}}

\newcommand{\vecH}{\boldsymbol{H}}

\newcommand{\vecM}{\boldsymbol{M}}

\newcommand{\vecZ}{\boldsymbol{Z}}

\newcommand{\veca}{\boldsymbol{a}}
\newcommand{\vecb}{\boldsymbol{b}}
\newcommand{\vecc}{\boldsymbol{c}}

\newcommand{\vece}{\boldsymbol{e}}

\newcommand{\vecg}{\boldsymbol{g}}

\newcommand{\vech}{\boldsymbol{h}}
\newcommand{\veck}{\boldsymbol{k}}

\newcommand{\vecn}{\boldsymbol{n}}

\newcommand{\vecu}{\boldsymbol{u}}
\newcommand{\vecv}{\boldsymbol{v}}
\newcommand{\vecw}{\boldsymbol{w}}
\newcommand{\vecx}{\boldsymbol{x}}

\newcommand{\vecphi}{\boldsymbol{\phi}}
\newcommand{\vecpsi}{\boldsymbol{\psi}}

\newcommand{\cB}{{\mathcal B}}

\newcommand{\cF}{{\mathcal F}}

\newcommand{\cL}{{\mathcal L}}

\newcommand{\cT}{{\mathcal T}}

\newcommand{\cX}{{\mathcal X}}

\newcommand{\cZ}{{\mathcal Z}}

\newcommand{\goto}{\rightarrow}
\newcommand{\gotoo}{\longrightarrow}
\newcommand{\mapstoo}{\longmapsto}
\newcommand{\p}{\partial}

\newcommand{\td}{\tilde}

\newcommand{\ds}{\, ds}
\newcommand{\dt}{\, dt}

\newcommand{\dvx}{\, d\vecx}

\newcommand{\nn}{\nonumber}

\numberwithin{equation}{section}
\newtheorem{theorem}{Theorem}[section]
\newtheorem{lemma}[theorem]{Lemma}
\newtheorem{corollary}[theorem]{Corollary}
\newtheorem{proposition}[theorem]{Proposition}

\newtheorem{definition}[theorem]{Definition}

\title{Existence of a unique solution and invariant measures for the stochastic Landau--Lifshitz--Bloch equation
 }
 \author{Zdzislaw Brze{\'z}niak, Beniamin Goldys and Kim Ngan Le}

\newcommand{\iprod}[1]{\langle#1\rangle}

\begin{document}

\maketitle
\tableofcontents
\pagenumbering{arabic}
\begin{abstract}\noindent
The Landau--Lifshitz--Bloch equation perturbed by a space-dependent noise was proposed in \cite{Garanin1991} as a model for evolution of spins in ferromagnatic materials at the full range of temperatures, including the temperatures higher than the Curie temperature. In the case of a ferromagnet filling a bounded domain $D\subset\R^d$, $d=1,2,3$, we show the existence of strong (in the sense of PDEs) martingale solutions. Furthermore, in cases $d=1,2$ we prove uniqueness of pathwise solutions and the existence of invariant measures\footnote{
The results of this paper have been presented at \textit{The conference on Stochastic Analysis and its Applications} B\'edlewo,  29th
May--3rd June 2017 
and at the \textit{AIMS Conference on  Dynamical Systems and Differential Equations}
Taipei, July 2018}
. 

\end{abstract}

\section{Introduction}
The aim of this paper is to initiate the analysis of stochastic Landau-Lifschitz-Bloch equation \eqref{eq: sLLB2}. For the reader's convenience we recall here some background material introduced in \cite{Le2016}. 
\par
A well-known model of ferromagnetic material leads to 
the Landau--Lifshitz--Gilbert  equation (LLGE) for the evolution of magnetic moment, which is valid only 
for temperatures close to the Curie temperature $T_{\text{c}}$~\cite{Gil55,LL35}. 
Several recent technological applications such as heat-assisted magnetic recording~\cite{Chantrelletal2015}, 
thermally assisted magnetic random access memories~\cite{Prejbeanu2007} 
or spincaloritronics have shown 
the need to generalise this theory to higher temperatures.
For high temperatures, a thermodynamically consistent 
approach was introduced by Garanin~\cite{Garanin1991,Garanin97} who derived 
the Landau--Lifshitz--Bloch equation (LLBE) for ferromagnets. 
The LLBE  essentially interpolates between 
the LLGE  at low temperatures and 
the Ginzburg-Landau theory of phase transitions. It is valid not only 
below but also above the Curie temperature. 
Let  $\vecu(t,\vecx)\in\R^3$ be the average spin polarisation
for $t>0$ and $\vecx\in D\subset \R^d$, $d=1,2,3$. The LLBE takes the form
\begin{equation}\label{eq: LLB}
\frac{\partial \vecu}{\partial t} 
= 
\gamma \vecu\times \vecH_{\text{eff}} 
+
L_1\frac{1}{|\vecu|^2}(\vecu\cdot\vecH_{\text{eff}})\vecu
-
L_2\frac{1}{|\vecu|^2}\vecu\times(\vecu\times\vecH_{\text{eff}})\,,
\end{equation}
where the effective field $\vecH_{\text{eff}}$ is given by \eqref{eq_eh} below. 
Here, $|\cdot|$ is the Euclidean norm in $\R^3$, $\gamma>0$ is the gyromagnetic ratio, and $L_1$ 
and $L_2$ are the longitudial  and transverse  
damping parameters, respectively.
\par\bigskip
Nevertheless, the deterministic LLBE is insufficient to capture the 
dispersion of individual trajectories at high temperatures. For example, 
when the magnetization is quenched it should describe the loss of magnetization 
correlations in different sites of the sample. In the laser-induced dynamics, 
this is responsible for the slowing down of the magnetization recovery at high laser fluency 
as the system temperature decreases~\cite{Evans12}. 
Therefore, under these circumstances and according to Brown~\cite{Brown1979,  Brown63}, 
stochastic forms of the LLBE are discussed in~\cite{ Evans12,Garanin2004} where 
the LLBE is modified in order to 
incorporate random fluctuations into the dynamics of the magnetisation and 
to describe noise-induced transitions between equilibrium states of the ferromagnet.
\par\medskip
In this paper,  we consider the stochastic LLBE, introduced in~\cite{Garanin2004}, perturbing the effective field 
$\vecH_{\text{eff}}$ in~\eqref{eq: LLB} by a Gaussian noise.
Furthermore, we focus on 
a case  
in which the temperature 
$T$ is raised higher than $T_{\text{c}}$, and  as a consequence the longitudial $L_1$ and transverse $L_2$ damping parameters 
are equal. The effective field $\vecH_{\text{eff}}$ is given by
\begin{equation}\label{eq_eh}
 \vecH_{\text{eff}} = \Delta\vecu - 
 \frac{1}{\chi_{||}}
 \bigg(1+\frac{3}{5}\frac{T}{T-T_c}|\vecu|^2\bigg)\vecu, 
\end{equation}
where $\chi_{||}$ is the longitudinal susceptibility. 
Using the vector product identity 
$\veca\times(\vecb\times\vecc) = \vecb(\veca\cdot\vecc)-\vecc(\veca\cdot\vecb)$ where $(\cdot,\cdot)$ 
is the scalar product in $\R^3$, we obtain 
\[
\vecu\times(\vecu\times\vecH_{\text{eff}})
=
(\vecu\cdot\vecH_{\text{eff}})\vecu
-
|\vecu|^2\vecH_{\text{eff}},
\]
and from property $L_1=L_2=:\kappa_1$, 
the stochastic LLBE takes the form
\begin{equation}\label{eq: sLLB2}
d\vecu 
= 
\bigl(
\kappa_1\Delta\vecu
+
\gamma\vecu\times  \Delta\vecu
-
\kappa_2(1+\mu|\vecu|^2)\vecu
\bigr)\dt
+
\sum_{k=1}^\infty
(\gamma\vecu\times\vech_k+\kappa_1\vech_k)\circ dW_k(t),
\end{equation}
with $\kappa_2 := \frac{\kappa_1}{\chi_{||}}$ and $\mu := \frac{3T}{5(T-T_c)}$.
Here, we assume that 
\begin{equation}\label{eq: condh}
\sum_{k=1}^\infty \|\vech_k\|_{\mW^{1,\infty}(D)}^2 \leq h <\infty,
\end{equation}
and $\{W_k:k\geq 1\}$ is a family of independent real-valued Wiener processes.
Finally, the stochastic LLBE being studied in this paper is equation~\eqref{eq: sLLB2} 
with   real positive 
coefficients $\kappa_1,\kappa_2,\gamma,\mu$, initial data 
$\vecu(0,\vecx)=\vecu_0(\vecx)$ and subject to 
homogeneous Neumann boundary conditions. 
\par
We emphasise that introducing two kinds of noise, multiplicative and additive, seems necessary to capture important features of the physical system. Namely, it is argued in \cite{Evans12} that only then the model may lead to a Boltzmann distribution valid for the full range of temperatures. 
\par\medskip
Despite its importance, very little is known about solutions to the deterministic and stochastic LLBE. A pioneering work on the existence of weak solutions to the  deterministic LLBE ~\eqref{eq: LLB} in a bounded
domain is carried out in~\cite{Le2016}. In this paper a Faedo--Galerkin approximation was introduced and 
the method of compactness was used to prove 
the existence of a weak solution for the LLBE and its regularity properties. In this work we built on the theory developed in ~\cite{Le2016} and initiated the theory of stochastic LLBE. While preparing its final version we learnt about the paper \cite{Jiang2019}. In their work the authors, starting from the formulation in ~\cite{Le2016}, prove the existence of weak (in PDE sense) martingale solutions to equation ~\eqref{eq: sLLB2}. In our work we show that martingale solutions are strong in PDE sense for $d=1,2,3$ and prove pathwise uniqeness in dimensions $d=1,2$ and this fact by the Yamada-Watanabe theorem implies uniqueness of martingale solutions. Finally, we prove the existence of an invariant measure which is an important step towards thermodynamic justification of the stochastic LLBE. The results of this paper have been presented at a number of international meetings (see footnote on p. 1). 
\par\bigskip
The paper is organized as follows. Section~\ref{sec: nota} contains  Theorem~\ref{theo: main} and Theorem~\ref{theo: pathwise} 
 on the existence and uniqueness strong solution of~\eqref{eq: sLLB2} 
as well as its regularity properties. In Section~\ref{sec: FG} we introduce the Faedo--Galerkin 
approximations and prove for their solutions some uniform bounds in various norms. 
Sections~\ref{sec: tight} and~\ref{sec: exist} are devoted to the proof of Theorem~\ref{theo: main}. 
The existence of an invariant measure stated in Theorem~\ref{theo: invariantmeasure} is proved in Section~\ref{sec: invariant}. 
Finally, in the Appendix we collect, 
for the reader's convenience, some facts scattered in the literature that are used 
in the course of the proof.
\section{Notation and the formulation of the main results}\label{sec: nota}
Let $D\subset\mathbb R^d$, $d=1,2,3$, be an open  bounded domain with uniformly $C^2$ boundary. The function space $\mH^1\!:=\mH^1(D,\R^3)$  is defined as follows:
\begin{align*}
\mH^1(D,\R^3)
&=
\left\{ \vecu\in\mL^2(D,\R^3) : \frac{\p\vecu}{\p
x_i}\in \mL^2(D,\R^3)\quad\text{for } i=1,2,3.
\right\}.
\end{align*}
Here, $\mL^p\!:=\mL^p(D,\R^3)$ with $p\ge 1$ is the usual space of $p^\text{th}$-power 
Lebesgue
integrable functions defined on $D$ and taking values in $\R^3$. 
Throughout this paper, we denote a scalar product in a Hilbert space $H$ by
$\inpro{\cdot}{\cdot}_H$ and its associated norm by $\|\cdot\|_H$. The duality between 
a space $X$ and its dual $X^*$ will be denoted by $_{X}\!\iprod{\cdot,\cdot}_{X^*}$. 
\par\medskip\noindent
Let $X_w$ denote the Hilbert space $X$ endowed with the  weak topology and let 
\begin{align*}
 C([0,T];X_w):= &\text{ the space of weakly continuous functions } u:=[0,T]\goto X\\
 &\text{ endowed with the weakest topology such that for all } h\in X
 \text{ the mapping }\\
 &\quad C([0,T];X_w)\gotoo C([0,T];\R)\\
 &\quad \quad \quad \quad \quad u \mapstoo \inpro{u(\cdot)}{h}_X \quad\text{is continuous.}
\end{align*}
In particular, $u_n\goto u$ in $C([0,T];X_w)$ iff for all $h\in X$: 
\[
 \lim_{n\goto\infty}\sup_{t\in[0,T]}\bigl|\inpro{u_n(t)-u(t)}{h}_X\bigr| = 0.
\]
For a ball $\mB^1(R):=\{\vecx\in\mH^1:\|\vecx\|_{\mH^1}\leq R\}$ we denote by $\mB^1_w(R)$ the ball 
$\mB^1(R)$ endowed with the weak topology. It is well known that $\mB^1_w(R)$ is metrizable~\cite{Brezis1983}. Let us consider the following subspace 
of $ C([0,T];\mH^1_w)$
\[
 C([0,T];\mB^1_w(R)) = \{ \vecu\in C([0,T];\mH^1_w) : \sup_{t\in[0,T]}\|\vecu\|_{\mH^1}\leq R \}.
\]
The space $\bigl(C([0,T];\mB^1_w),\rho\bigr)$ is a   complete metric space with 
\[
 \rho(\vecu,\vecv) = \sup_{t\in[0,T]}q(\vecu(t),\vecv(t)),
\]
where $q$ is the metric compatible with the weak topology on $\mB^1$.
\begin{definition}\label{def: weakso}
Let $d=1,2,3$. Given $\vecu_0\in\mH^1$ and $T>0$, a weak martingale solution $(\Omega,\cF,\mF,\mP,W,\vecu)$
to~\eqref{eq: sLLB2}, consists of 
\begin{enumerate}
\renewcommand{\labelenumi}{(\alph{enumi})}
\item
a filtered probability space
$(\Omega,\cF,\mF,\mP)$ with the
filtration $\mF=\left(\mathcal F_t\right)$ satisfying the usual conditions,
\item
a family of independent real-valued Wiener processes 
$W=(W_j)_{j=1}^\infty $, adapted to the filtration $\mF$, 
\item 
a progressively measurable
process $\vecu : [0,T]\times \Omega \goto \mH^1\cap \mL^4$
such that $\mP$-a.s. $\vecu\in\mL^\infty\left(0,T;\mH^1\right)$ and 
for every  $t\in[0,T]$ and $\vecphi\in \C_0^{\infty}(D)$, $\mP$-a.s.:
\end{enumerate}
\begin{align}\label{eq: weakLLB}
\iprod{\vecu(t),\vecphi}_{\mL^2} 
=
&\iprod{\vecu_0,\vecphi}_{\mL^2}
-\kappa_1\int_0^t \iprod{\nabla\vecu(s),\nabla\vecphi}_{\mL^2}\ds 
-\gamma\int_0^t \iprod{\vecu(s)\times\nabla\vecu(s),\nabla\vecphi}_{\mL^2}\ds\nn \\
&-\kappa_2\int_0^t \iprod{(1+\mu|\vecu|^2(s))\vecu(s),\vecphi}_{\mL^2}\ds\nn\\
&+
\int_0^t
\sum_{k=1}^\infty
\iprod{
\gamma\vecu(s)\times\vech_k + \kappa_1\vech_k,\vecphi}_{\mL^2}\circ dW_k(s),
\end{align}
\end{definition}
Now we  can formulate the main results of this paper.
\begin{theorem}\label{theo: main}
Let $d=1,2,3$. 
Assume that $\|u_0\|_{\mH^1}<C_1$ for a certain $C_1>0$. Then
there exists a weak martingale solution $(\Omega,\cF,\mF,\mP,W,\vecu)$ of~\eqref{eq: sLLB2} such that 
\begin{enumerate}
\item for every $p\in[1,\infty)$, $\alpha\in(0,1/2)$
\begin{align}
\vecu\in L^{p}\bigl(\Omega;\bigl(L^{\infty}(0,T;\mH^1)\cap L^2(0,T;\mH^2)\bigr)\bigr),\nn\\
\mE\|\vecu\|^q_{W^{\gamma,p}(0,T;\mL^2)} + \mE\|\vecu\|^p_{L^{\infty}(0,T;\mH^1)\cap L^2(0,T;\mH^2)}  < c, \label{eq_est1}
\end{align}
and for every $q\in\left[1,\frac43\right)$ 
\begin{equation}\label{eq_est2}
\mE\left(\int_0^T\|\vecu(t)\times\Delta \vecu(t)\|_{\mL^{2}}^q \,dt\right)^p< c\,,
\end{equation}
where c is a positive constant depending on $p$, $C_1$ and $h$.
 \item the following equality holds in $\mL^2$:
\begin{align}\label{eq: mainu}
\vecu(t) 
&=
\vecu(0)
+
\kappa_1\int_0^t \Delta\vecu(s)\ds 
+
\gamma \int_0^t \vecu(s)\times\Delta\vecu(s)\ds\nn\\
&\quad-
\kappa_2 \int_0^t (1+\mu|\vecu|^2)\vecu(s)\ds  
+
\sum_{k=1}^\infty
\int_0^t
\bigl(\gamma\vecu(s)\times\vech_k + \kappa_1\vech_k\bigr)\circ dW_k(s),
\end{align}
\item 
for every $\bar\alpha\in\left[0,\frac14\right)$ and $\beta\in [0,\frac12]$
\begin{equation}\label{eq: mainu2}
\vecu(\cdot)\in C^{\bar\alpha}([0,T];\mL^{2})\cap C^\beta([0,T];\mL^{3/2}) \cap C([0,T];\mH^1_w)\bigr)\quad \mP-a.s. 
\end{equation}
\end{enumerate}
\end{theorem}
\begin{theorem}\label{theo: pathwise}(Pathwise uniqueness)
Let $D\subset \R$ or $D\subset \R^2$ and let $\vecu_0\in\mH^1$ be fixed. Assume that 
$(\Omega,\cF,\mF,\mP, W,\vecu_1)$ and $(\Omega,\cF,\mF,\mP, W,\vecu_2)$ are two weak martingale solutions to equation ~\eqref{eq: mainu}, such that for $i=1,2$
\begin{enumerate}
\renewcommand{\labelenumi}{(\alph{enumi})}
\item 
$\vecu_1(0)=\vecu_2(0)=\vecu_0$,
 \item the paths of $\vecu_i$ lie in $L^{\infty}(0,T;\mH^1)\cap L^2(0,T;\mH^2)$,
 \item each $\vecu_i$ satisfies equation~\eqref{eq: mainu}.
\end{enumerate}
Then, for $\mP$-a.e. $\omega\in\Omega$
\[
 \vecu_1(\cdot,\omega) = \vecu_2(\cdot,\omega).
\]
\end{theorem}
\begin{proof}
 Let $\vecv:=\vecu_1-\vecu_2$. Then $\vecv$ satisfies the following equation
\begin{align}\label{eq: uni1}
d\vecv 
=
\bigl[
\kappa_1\Delta\vecv
&+
\gamma(\vecv\times\Delta\vecu_1+\vecu_2\times\Delta\vecv
-
\kappa_2\bigl(\vecv+\mu(|\vecu_1|^2\vecv + (|\vecu_1|^2-|\vecu_2|^2)\vecu_2)\bigr)\nn\\
&+
\frac12\gamma\sum_{k=1}^\infty(\vecv\times\vech_k)\times\vech_k
\bigr]\dt
+
\gamma\sum_{k=1}^\infty\vecv\times\vech_k\,dW_k(t),
\end{align}
with $\vecv(\cdot,0)= \vecv_0=0$
By using It\^o Lemma and~\eqref{eq: uni1} we get 
\begin{align*}
\frac12 d\|\vecv\|^2
&=
\inpro{\vecv}{d\vecv}_{\mL^2}
+
\frac12\inpro{d\vecv}{d\vecv}_{\mL^2}\\
&=
\bigl[
-\kappa_1\|\nabla\vecv\|_{\mL^2}^2
+\gamma \inpro{\vecv}{\vecu_2\times\Delta\vecv}
-\kappa_2\|\vecv\|^2_{\mL^2}
-\kappa_2\mu\||\vecu_1|\vecv\|^2_{\mL^2}\\
&\quad-
\kappa_2\mu\inpro{\vecv}{(|\vecu_1|^2-|\vecu_2|^2)\vecu_2}_{\mL^2}
-
\frac12 \sum_{k=1}^\infty \|\vecv\times\vech_k\|^2_{\mL^2}
\bigr]\dt,
\end{align*}
hence
\begin{align}\label{eq: uni2}
\frac12\|\vecv(t)\|^2
+
\kappa_1\int_0^t\|\nabla\vecv(s)\|_{\mL^2}^2\ds
+
\kappa_2\|\vecv(t)\|^2_{\mL^2}
\leq 
&\frac12\|\vecv_0\|^2
+
\gamma \int_0^t\inpro{\vecv(s)}{\vecu_2(s)\times\Delta\vecv(s)}_{\mL^2}\ds\nn\\
&-
\kappa_2\mu\int_0^t\inpro{\vecv(s)}{(|\vecu_1(s)|^2-|\vecu_2(s)|^2)\vecu_2(s)}_{\mL^2}\ds.
\end{align}
We now estimate all terms in the right hand side of~\eqref{eq: uni2}. 

Let us start with the second term. By using the triangle inequality, there holds
\begin{align}\label{eq: uni3}
\bigl|\int_0^t&\inpro{\vecv(s)}{(|\vecu_1(s)|^2-|\vecu_2(s)|^2)\vecu_2(s)}_{\mL^2}\ds\bigr|
\leq 
\int_0^t\int_D
|\vecv|^2 |\vecu_2|(|\vecu_1|+|\vecu_2|)\dvx\ds\nn\\
&\quad\leq
\int_0^t
\|\vecu_2(s)\|_{\mL^{\infty}}(\|\vecu_1(s)\|_{\mL^{\infty}}+\|\vecu_2(s)\|_{\mL^{\infty}})
\|\vecv(s)\|_{\mL^2}^2
\ds
\end{align}

The first term in case $D\subset \R$ is estimated by noting the following interpolation inequality
\[
\|\vecv\|_{\mL^{\infty}} 
\leq 
c\|\vecv\|^{\frac12}_{\mL^2}\|\vecv\|^{\frac12}_{\mH^1},
\]
there hold
\begin{align}\label{eq: uni4}
\bigl|\int_0^t\inpro{\vecv(s)}{\vecu_2(s)\times\Delta\vecv(s)}_{\mL^2}\ds\bigr|
&= 
\bigl|\int_0^t\inpro{\vecv(s)}{\nabla\vecu_2(s)\times\nabla\vecv(s)}_{\mL^2}\ds\bigr|\nn\\
&\leq 
\int_0^t
\|\vecv(s)\|_{\mL^{\infty}} \|\nabla\vecv(s)\|_{\mL^2}
\|\nabla\vecu_2(s)\|_{\mL^2}
\ds\nn\\
&\leq 
c\int_0^t
\|\vecv\|^{\frac12}_{\mL^2}\|\vecv\|^{\frac12}_{\mH^1} \|\nabla\vecv\|_{\mL^2}
\|\nabla\vecu_2\|_{\mL^2}
\ds\nn\\
&\leq 
c\int_0^t
\|\vecv\|_{\mL^2}\|\nabla\vecv\|_{\mL^2}
\|\nabla\vecu_2\|_{\mL^2}
\ds\nn\\
&\quad+
c\int_0^t
\|\vecv\|_{\mL^2}^{1/2}\|\nabla\vecv\|_{\mL^2}^{3/2}
\|\nabla\vecu_2\|_{\mL^2}
\ds\nn\\
&\leq 
\frac{\kappa_1}{2\gamma}\int_0^t\|\nabla\vecv\|_{\mL^2}^2\ds\nn\\
&\quad+
\frac{\gamma c^2}{\kappa_1}
\int_0^t
\|\nabla\vecu_2\|_{\mL^2}^2
\bigl(
1+
\frac{4\gamma c^2}{\kappa_1}
\|\nabla\vecu_2\|_{\mL^2}^2
\bigr)
\|\vecv\|_{\mL^2}^2\ds.
\end{align}
It follows from~\eqref{eq: uni2}--\eqref{eq: uni4} that 
\[
\|\vecv(t)\|_{\mL^2}^2
\leq 
\frac12\|\vecv_0\|_{\mL^2}^2
+
c\int_0^t
\Phi(s)\|\vecv(s)\|_{\mL^2}^2\ds,
\]
where 
\[
\Phi(s)
:= 
\|\vecu_2(s)\|_{\mL^{\infty}}(\|\vecu_1(s)\|_{\mL^{\infty}}+\|\vecu_2(s)\|_{\mL^{\infty}})
+
\|\nabla\vecu_2(s)\|_{\mL^2}^2 + \|\nabla\vecu_2(s)\|_{\mL^2}^4.
\]
Using Gronwall inequality and noting $\int_0^t\Phi(s)\ds<\infty$, we deduce that 
\[
\|\vecv(t)\|_{\mL^2}^2
\leq 
e^{c\int_0^t\Phi(s)\ds}\|\vecv_0\|_{\mL^2}^2,
\]
it implies $\vecu_1(\cdot,\omega)=\vecu_2(\cdot,\omega)$ for $\mP$-a.e. $\omega\in\Omega$ as $\vecv_0=0$.

The first term in case $D\subset \R^2$ is estimated by noting that 
\[
\|\vecw\|_{\mL^4}\leq c \|\vecw\|_{\mL^2}^{\frac12}\|\vecw\|_{\mH^1}^{\frac12},
\quad\forall \vecw\in \mH^1(D)\text{ and } D\subset \R^2
\]
and using the Young inequality $ab\leq \frac{a^{4/3}}{4/3} +\frac{b^4}{4}$, there hold
\begin{align}\label{eq: uni5}
\bigl|\int_0^t\inpro{\vecv(s)}{\vecu_2(s)\times\Delta\vecv(s)}_{\mL^2}\ds\bigr|
&= 
\bigl|\int_0^t\inpro{\vecv(s)}{\nabla\vecu_2(s)\times\nabla\vecv(s)}_{\mL^2}\ds\bigr|\nn\\
&\leq 
\int_0^t
\|\vecv(s)\|_{\mL^4} \|\nabla\vecv(s)\|_{\mL^2}
\|\nabla\vecu_2(s)\|_{\mL^4}
\ds\nn\\
&\leq 
c\int_0^t
\|\vecv\|^{\frac12}_{\mL^2}\|\vecv\|^{\frac12}_{\mH^1} \|\nabla\vecv\|_{\mL^2}
\|\nabla\vecu_2\|_{\mL^4}
\ds\nn\\
&\leq 
c\int_0^t
\|\vecv\|_{\mL^2}^{1/2}\|\nabla\vecv\|_{\mL^2}^{3/2}
\|\nabla\vecu_2\|_{\mL^4}
\ds
+ 
c\int_0^t
\|\vecv\|_{\mL^2}\|\nabla\vecv\|_{\mL^2}
\|\nabla\vecu_2\|_{\mL^4}
\ds
\nn\\
&\leq 
\frac{3\kappa_1}{4\gamma}\int_0^t\|\nabla\vecv\|_{\mL^2}^2\ds
+
\frac{\gamma^3 c^2}{4\kappa_1^3}
\int_0^t
\|\nabla\vecu_2\|_{\mL^4}^4
\|\vecv\|_{\mL^2}^2\ds\nn\\
&\quad +
\frac{\kappa_1}{8\gamma}\int_0^t\|\nabla\vecv\|_{\mL^2}^2\ds
+
\frac{\gamma c^2}{2\kappa_1^2}
\int_0^t
\|\nabla\vecu_2\|_{\mL^4}^2
\|\vecv\|_{\mL^2}^2\ds\nn\\
&\leq 
\frac{5\kappa_1}{8\gamma}\int_0^t\|\nabla\vecv\|_{\mL^2}^2\ds\nn\\
&\quad+
c
\int_0^t
\big(\|\nabla\vecu_2\|_{\mL^2}^2\|\vecu_2\|_{\mH^2}^2
+\|\nabla\vecu_2\|_{\mL^2}\|\vecu_2\|_{\mH^2}
\big)
\|\vecv\|_{\mL^2}^2\ds.
\end{align}
It follows from~\eqref{eq: uni2},~\eqref{eq: uni3} and~\eqref{eq: uni5} that 
\[
\|\vecv(t)\|_{\mL^2}^2
\leq 
\frac12\|\vecv_0\|_{\mL^2}^2
+
c\int_0^t
\Phi(s)\|\vecv(s)\|_{\mL^2}^2\ds,
\]
where 
\[
\Phi(s)
:= 
\|\vecu_2(s)\|_{\mL^{\infty}}(\|\vecu_1(s)\|_{\mL^{\infty}}+\|\vecu_2(s)\|_{\mL^{\infty}})
+
\|\nabla\vecu_2\|_{\mL^2}^2\|\vecu_2\|_{\mH^2}^2 + \|\nabla\vecu_2\|_{\mL^2}\|\vecu_2\|_{\mH^2}.
\]
Using the Gronwall inequality and noting $\int_0^t\Phi(s)\ds<\infty$, we find that 
\[
\|\vecv(t)\|_{\mL^2}^2
\leq 
e^{c\int_0^t\Phi(s)\ds}\|\vecv_0\|_{\mL^2}^2,
\]
it implies $\vecu_1(\cdot,\omega)=\vecu_2(\cdot,\omega)$ for $\mP$-a.e. $\omega\in\Omega$ as $\vecv_0 = 0$.
\end{proof}
\begin{corollary}\label{co: strongsolution}
Let $D\subset \R$ or $D\subset \R^2$. Then for every $\vecu_0\in\mH^1$
\begin{enumerate}
 \item there exists a pathwise unique strong solution of equation~\eqref{eq: sLLB2};
 \item the martingale solution of~\eqref{eq: sLLB2} is unique in law.
\end{enumerate}
\end{corollary}
\begin{proof}
Since by Theorem~\ref{theo: main} there exists a martingale solution and by Theorem~\ref{theo: pathwise} it is pathwise 
unique, the corrollary follows from Theorem 2.2 and 12.1 in~\cite{Ondrejat2004}.
\end{proof}

\section{Faedo-Galerkin Approximation}\label{sec: FG}
Let $A=-\Delta$ be the negative Neumann Laplacian in $D$. 
Then ~\cite[Theorem 1, p. 335]{Evans1998}, there exists an orthonormal basis 
$\{\vece_i\}_{i=1}^{\infty}$ of $\mL^2$, consisting 
of eigenvectors of $A$, such that for all $i=1,2,\ldots$ 
\[
-\Delta\vece_i = \lambda_i\vece_i,\quad 
\frac{\partial\vece_i}{\partial\vecn} = 0 \text{ on }\partial D,\quad \vece_i\in\mL^{\infty}\,,
\]
where $\vecn$ is the outward normal on the boundary $\partial D$; 
and $\lambda_i>0$ for $i=1,2,$\dots are eigenvalues of $A$. For $\beta>0$ we define the Hilbert space $X^\beta=\mathrm{dom}\left(A^\beta\right)$ endowed with the norm 
\[\|\vecv\|_{X^\beta}=\left\|(I+A)^\beta\vecv\right\|_{\mL^2}\,.\]
The dual space will be denoted by $X^{-\beta}$. \\
Let $S_n:=\text{span}\{\vece_1,\cdots,\vece_n\}$ and 
$\Pi_n$ be the orthogonal projection from $\mL^2$ onto $S_n$, 
defined by: for $\vecv\in\mL^2$
\begin{equation}\label{eq: Pi_n}
\iprod{\Pi_n\vecv,\vecphi}_{\mL^2} 
= \iprod{\vecv,\vecphi}_{\mL^2},\quad\forall \vecphi\in S_n.
\end{equation}
We note that  
\begin{equation}\label{eq: adj2}
\iprod{\nabla\vecw,\nabla\Pi_n\vecv}_{\mL^2}
=
\iprod{\vecv,A\Pi_n\vecw}_{\mL^2}
=
\iprod{\nabla\vecv,\nabla\Pi_n\vecw}_{\mL^2}, 
\end{equation}
for $\vecv,\vecw\in\mH^1$, 
hence 
\begin{equation}\label{eq: boundPi_n}
\|\Pi_n\vecv\|_{\mL^2}
\leq
\|\vecv\|_{\mL^2}
\quad\forall\vecv\in \mL^2,
\end{equation}
and 
\begin{equation}\label{eq: boundPi_n2}
\|\nabla\Pi_n\vecv\|_{\mL^2}
\leq \|\nabla\vecv\|_{\mL^2}
\quad\forall\vecv\in \mH^1.
\end{equation}

We are now looking for approximate solution 
$\vecu_n(t)\in S_n:=\text{span}\{\vece_1,\cdots,\vece_n\}$ of equation~\eqref{eq: sLLB2} satisfying
\begin{align}\label{eq: GaLLB}
d\vecu_n 
= 
&\bigl(
\kappa_1\Delta\vecu_n
+
\gamma\Pi_n(\vecu_n\times  \Delta\vecu_n)
-
\kappa_2\Pi_n\bigl((1+\mu|\vecu_n|^2)\vecu_n\bigr)
\bigr)\dt\nn\\
&+
\sum_{k=1}^n
\Pi_n\bigl(\gamma
\vecu_n\times\vech_k + \kappa_1\vech_k\bigr)\circ dW_k(t),
\end{align}
with $\vecu_n(0) = \vecu_{0n}=\Pi_n\vecu_0$. 
The existence of a local solution to~\eqref{eq: GaLLB} 
is a consequence of the following lemma.
\begin{lemma}
For $n\in\N$, define the maps:
\begin{align*}
&F^1_n: S_n\ni \vecv \mapsto \Delta\vecv\in S_n,\\
&F^2_n: S_n\ni\vecv\mapsto \Pi_n(\vecv\times\Delta\vecv)\in S_n,\\
&F^3_n: S_n\ni\vecv\mapsto \Pi_n((1+\mu|\vecv|^2)\vecv)\in S_n,\\
&G_{nk} : S_n\ni\vecv\mapsto \Pi_n\bigl(
\gamma\vecu_n\times\vech_k+ \kappa_1\vech_k\bigr)
\end{align*}
Then $F^1_n$ and $G_{nk}$ are globally Lipschitz and $F^2_n$, $F^3_n$ are locally Lipschitz.
\end{lemma}
\begin{proof}
For any $\vecv\in S_n$ we have 
\begin{equation*}
\vecv = \sum_{i=1}^n \inpro{\vecv}{\vece_i}_{\mL^2}\vece_i
\quad\text{and}\quad
-\Delta\vecv = \sum_{i=1}^n \lambda_i\inpro{\vecv}{\vece_i}_{\mL^2}\vece_i.
\end{equation*}
Using the triangle inequality and the H\"older inequality, we obtain for any $\vecu,\vecv\in S_n$ 
\begin{align*}
\|F^1_n(\vecu)-F^1_n(\vecv)\|_{\mL^2}
&=
\|\Delta\vecu-\Delta\vecv\|_{\mL^2}
=
\|\sum_{i=1}^n \lambda_i\inpro{\vecu-\vecv}{\vece_i}_{\mL^2}\vece_i\|_{\mL^2}\\
&\leq
\sum_{i=1}^n \lambda_i\bigl|\inpro{\vecu-\vecv}{\vece_i}_{\mL^2}\bigr|
\leq
\left(\sum_{i=1}^n \lambda_i\right)\|\vecu-\vecv\|_{\mL^2},
\end{align*}
and the globally Lipschitz property of $F^1_n$ follows immediately.\\
Next, estimate ~\eqref{eq: boundPi_n}  yields 
\begin{align*}
\|F^2_n(\vecu)-F^2_n(\vecv)\|_{\mL^2}
&=
\|\Pi_n(\vecu\times\Delta\vecu-\vecv\times\Delta\vecv)\|_{\mL^2}
\leq
\|\vecu\times\Delta\vecu-\vecv\times\Delta\vecv\|_{\mL^2}\\
&\leq
\|\vecu\times(\Delta\vecu-\Delta\vecv)\|_{\mL^2}
+
\|(\vecu-\vecv)\times\Delta\vecv\|_{\mL^2}\\
&\leq
\|\vecu\|_{\mL^{\infty}}
\|F^1_n(\vecu)-F^1_n(\vecv)\|_{\mL^2}
+
\|(\vecu-\vecv)\|_{\mL^2}
\|\Delta\vecv\|_{\mL^{\infty}}.
\end{align*}
Since $F^1_n$ is globally Lipschitz and all norms on the finite dimensional space $S_n$ are equivalent, $F^2_n$ is locally Lifshitz.\\
Similarly, the local Lipschitz property of $F^3_n$ follows from the estimate,
\begin{align*}
\|F^3_n(\vecu)-F^3_n(\vecv)\|_{\mL^2}
&\leq
\|\vecu-\vecv\|_{\mL^2}
+
\mu\|\Pi_n(|\vecu|^2\vecu-|\vecv|^2\vecv)\|_{\mL^2}\\
&\leq
\|\vecu-\vecv\|_{\mL^2}
+
\mu\||\vecu|^2\vecu-|\vecv|^2\vecv\|_{\mL^2}\\
&\leq
\|\vecu-\vecv\|_{\mL^2}
+
\mu\||\vecu|^2(\vecu-\vecv)\|_{\mL^2}
+
\mu\|(\vecu-\vecv)\cdot(\vecu+\vecv)\,\vecv\|_{\mL^2}\\
&\leq
\bigl(
1
+
\mu\||\vecu|^2\|_{\mL^{\infty}}
+
\mu\|\vecu+\vecv\|_{\mL^{\infty}}\|\vecv\|_{\mL^{\infty}}\bigr)
\|\vecu-\vecv\|_{\mL^2}.
\end{align*}
which completes the proof of this lemma.
\end{proof}
\noindent
We first recall the relation between the Stratonovich and It\^o differentials: if $W_k$ is an $\R$-valued standard Wiener process defined on a certain filtered probability space $(\Omega, \mathcal F,\mF,\mP)$ then 
\begin{equation*}
G_{nk}(\vecv)\circ dW_k(t)
=
\frac12 G_{nk}'(\vecv)[G_{nk}(\vecv)]\dt
+
G_{nk}(\vecv)\, dW_k(t),
\end{equation*}
where 
\[
G_{nk}'(\vecv)[G_{nk}(\vecv)] 
= \Pi_n\bigl(G_{nk}(\vecv)\times\vech_k\bigr)\,.
\]
Therefore ~\eqref{eq: GaLLB} can be rewritten as an It\^o equation 
\begin{equation}\label{eq: GaLLB2}
d\vecu_n = F_n(\vecu_n)\dt + \sum_{k=1}^n G_{nk}(\vecu_n)dW_k(t),
\end{equation}
with 
\[
F_n(\vecu_n) 
:= 
F^1_n(\vecu_n) 
+ F^2_n(\vecu_n) 
- F^3_n(\vecu_n) 
+ \frac12\sum_{k=1}^n\Pi_n\bigl(G_{nk}(\vecu_n)\times\vech_k\bigr).
\]
We now proceed to prove uniform bounds for the approximate solutions $\vecu_n$.
\begin{lemma}\label{lem: appSo_sta1}
For any $p\geq 1$, $n=1,2,$\dots and every $t\in[0,\infty)$, there holds 
\begin{equation*}
\mE \sup_{s\in[0,t]} \|\vecu_n(s)\|^{2p}_{\mL^2} 
+
\mE\left(\int_0^t
\|\nabla\vecu_n\|^2_{\mL^2}\ds\right)^p
+
\mE\left(\int_0^t
\int_D
\bigl(1+ \mu|\vecu_n|^2\bigr)|\vecu_n|^2\dvx\ds\right)^p
\leq
c,
\end{equation*}
where c is a positive constant depending on $p$, $C_1$ and $h$.
\end{lemma}
\begin{proof}
Let us consider a function $\psi: S_n\ni\vecv\mapsto \frac12\|\vecv\|^2_{\mL^2}\in \R$ that is $C^\infty$ with 
\begin{equation*}
\psi'(\vecv)(\vecg)  = \inpro{\vecv}{\vecg}_{\mL^2}
\quad\text{and}\quad
\psi''(\vecv)(\vecg,\veck)  = \inpro{\veck}{\vecg}_{\mL^2}
\quad\text{for all }\vecg,\veck\in S_n\,.
\end{equation*}
Using the It\^o Lemma we obtain 
\begin{equation}\label{eq: itopsi}
d\psi(\vecu_n) =   \inpro{\vecu_n}{d\vecu_n}_{\mL^2} + \frac12\inpro{d\vecu_n}{d\vecu_n}_{\mL^2}.
\end{equation}
From~\eqref{eq: GaLLB2} and~\eqref{eq: itopsi}  we deduce that 
\begin{align}\label{eq: ito1}
\frac12 d\|\vecu_n(t)\|^2_{\mL^2} 
&=
\bigl(
\inpro{\vecu_n}{F_n(\vecu_n)}_{\mL^2} 
+
\frac12 \sum_{k=1}^n \|G_{nk}(\vecu_n)\|^2_{\mL^2}
\bigr)\dt
+
\sum_{k=1}^n \inpro{\vecu_n}{G_{nk}(\vecu_n)}_{\mL^2}\,dW_k(t).
\end{align}
We also have 
\begin{align}
\inpro{\vecu_n}{F_n(\vecu_n)}_{\mL^2} 
&=
-\kappa_1\|\nabla \vecu_n(t)\|^2_{\mL^2}
-\kappa_2\int_D \bigl(1+ \mu|\vecu_n|^2\bigr)|\vecu_n|^2\dvx\nn\\
&\quad+\frac12 \sum_{k=1}^n \gamma
\inpro{\vecu_n}{G_{nk}(\vecu_n)\times \vech_k}_{\mL^2}\nn \\
&=
-\kappa_1\|\nabla \vecu_n(t)\|^2_{\mL^2}
-\kappa_2\int_D \bigl(1+ \mu|\vecu_n|^2\bigr)|\vecu_n|^2\dvx\nn\\
&\quad
-\frac12\sum_{k=1}^n \|G_{nk}(\vecu_n)\|^2_{\mL^2}
+\frac12\kappa_1 \sum_{k=1}^n 
\inpro{\vech_k}{G_{nk}(\vecu_n)}_{\mL^2},\label{eq: ito2}\\
\text{and }\quad
\inpro{\vecu_n}{G_{nk}(\vecu_n)}_{\mL^2}
&=
\kappa_1\inpro{\vecu_n}{\vech_k}_{\mL^2}.\label{eq: ito3}
\end{align}
Therefore, taking into account that ~\eqref{eq: ito1}--\eqref{eq: ito3} and~\eqref{eq: boundPi_n} we obtain 
\begin{align}
\frac12 \|\vecu_n(t)\|^2_{\mL^2} 
&+\kappa_1\int_0^t\|\nabla \vecu_n(s)\|^2_{\mL^2}\ds
+\kappa_2\int_0^t\int_D \bigl(1+ \mu|\vecu_n(s)|^2\bigr)|\vecu_n|^2(s)\dvx\ds\nn\\
&=
\frac12 \|\vecu_{0,n}\|^2_{\mL^2} 
+
\frac12\kappa_1 \sum_{k=1}^n \int_0^t
\inpro{\vech_k}{G_{nk}(\vecu_n(s))}_{\mL^2}\ds\nn\\
&\quad+\kappa_1
\sum_{k=1}^n \int_0^t \inpro{\vecu_n(s)}{\vech_k}_{\mL^2}\,dW_k(s)\nn\\
&\leq 
\frac12 \|\vecu_{0,n}\|^2_{\mL^2} 
+
ct\sum_{k=1}^n \|\vech_k\|^2_{\mL^2} 
+
c\bigl(\sum_{k=1}^n \|\vech_k\|^2_{\mL^{\infty}}\bigr)
\int_0^t\|\vecu_n(s)\|_{\mL^2}\ds\nn\\
&\quad+
c
\sum_{k=1}^n \int_0^t \inpro{\vecu_n(s)}{\vech_k}_{\mL^2}\,dW_k(s)\nn\\
&\leq 
c
+c
\int_0^t\|\vecu_n(s)\|^2_{\mL^2}\ds
+
c
\sum_{k=1}^n \int_0^t \inpro{\vecu_n(s)}{\vech_k}_{\mL^2}\,dW_k(s).\label{eq: ito4}
\end{align}
It follows from~\eqref{eq: ito4} and Jensen's inequality that for any $p\geq 1$ there hold
\begin{align}\label{eq: ito41}
 \|\vecu_n(t)\|^{2p}_{\mL^2} 
 &+
 \bigl(\int_0^t\|\nabla \vecu_n(s)\|^2_{\mL^2}\ds\bigr)^p
 +
 \bigl(\int_0^t\int_D \bigl(1+ \mu|\vecu_n(s)|^2\bigr)|\vecu_n|^2(s)\dvx\ds\bigr)^p\nn\\
&\leq 
 c
 +c
\bigl(\int_0^t\|\vecu_n(s)\|^2_{\mL^2}\ds\bigr)^p
+
c
\bigl|\sum_{k=1}^n \int_0^t \inpro{\vecu_n(s)}{\vech_k}_{\mL^2}\,dW_k(s)\bigr|^p\nn\\
&\leq 
c
+c
\int_0^t\|\vecu_n(s)\|^{2p}_{\mL^2}\ds
+
c
\bigl|\sum_{k=1}^n \int_0^t \inpro{\vecu_n(s)}{\vech_k}_{\mL^2}\,dW_k(s)\bigr|^p
\end{align}
Using the Burkholder-Davis-Gundy inequality and the H\"older inequality, we estimate
\begin{align*}
\mE \sup_{s\in[0,t]}
\bigl|\sum_{k=1}^n \int_0^s \inpro{\vecu_n(\tau)}{\vech_k}_{\mL^2}\,dW_k(\tau)\bigr|^p
&\leq
c\mE \bigl|\sum_{k=1}^n \int_0^t (\inpro{\vecu_n(s)}{\vech_k}_{\mL^2})^2\ds\bigr|^{p/2}\\
&\leq 
c\mE \bigl|\sum_{k=1}^n \int_0^t \|\vecu_n(s)\|_{\mL^2}^2\|\vech_k\|_{\mL^2}^2\ds\bigr|^{p/2}\\
&\leq 
c\bigl(\sum_{k=1}^n \|\vech_k\|^2_{\mL^2}\bigr)^{p/2}\mE \bigl[\bigl|\int_0^t \|\vecu_n(s)\|_{\mL^2}^2\ds\bigr|^{p/2}\bigr]\\
&\leq 
c\mE \bigl[ 1 + (\int_0^t \|\vecu_n(s)\|_{\mL^2}^2\ds)^p\bigr]\\
&\leq 
c + c\int_0^t \mE\|\vecu_n(s)\|_{\mL^2}^{2p}\ds,
\end{align*}
and in view of ~\eqref{eq: ito4}, we find that 
\begin{align*}
\mE\|\vecu_n(t)\|^{2p}_{\mL^2} 
+\mE\left(\int_0^t\|\nabla \vecu_n(s)\|^2_{\mL^2}\ds\right)^p
+&\mE\left(\int_0^t\int_D\left(1+\mu\left|\vecu_n(s)\right|^2\right)\left|\vecu_n(s)\right|^2\dvx\ds
\right)^p\\
&\leq 
c + c\int_0^t \mE\|\vecu_n(s)\|_{\mL^2}^{2p}\ds.
\end{align*}
In particular, for any $p\geq 1$ 
\begin{align*}
\mE\|\vecu_n(t)\|^{2p}_{\mL^2} 
\leq 
c + c\int_0^t \mE\|\vecu_n(s)\|_{\mL^2}^{2p}\ds.
\end{align*}
The result follows immediately from the Gronwall inequality, which completes the 
proof of this lemma.
\end{proof}
\begin{lemma}\label{lem: appSo_sta2}
For any $p\geq 1$, $n=1,2,$\dots and every $t\in[0,\infty)$, there holds 
\begin{equation*}
\mE \sup_{s\in[0,t]}  \|\nabla\vecu_n(s)\|^{2p}_{\mL^2} 
+
\kappa_2\mE\left(\int_0^t
\|\Delta\vecu_n(s)\|^2_{\mL^2}\ds\right)^p
\leq
c,
\end{equation*}
where c is a positive constant depending on $C_1$ and $h$.
\end{lemma}
\begin{proof}
In a similar fashion as in the proof of Lemma~\ref{lem: appSo_sta1}, we consider a function 
\[\psi: S_n\ni\vecv\mapsto \frac12\|\nabla\vecv\|^2_{\mL^2}\in \R.\]
By noting that for all $\vecg,\veck\in S_n$,
\begin{equation*}
\psi'(\vecv)(\vecg)  = \inpro{\nabla\vecv}{\nabla\vecg}_{\mL^2} = -\inpro{\Delta\vecv}{\vecg}_{\mL^2}
\quad\text{and}\quad
\psi''(\vecv)(\vecg,\veck)  = \inpro{\nabla\veck}{\nabla\vecg}_{\mL^2}.
\end{equation*}
and using the It\^o Lemma we get 
\begin{equation*}
d\psi(\vecu_n) 
=   
-\inpro{\Delta\vecu_n}{d\vecu_n}_{\mL^2} + \frac12\inpro{\nabla d\vecu_n}{\nabla d\vecu_n}_{\mL^2}.
\end{equation*}
This equation together with~\eqref{eq: GaLLB2} and~\eqref{eq: itopsi} yields
\begin{align}\label{eq: ito21}
\frac12 d\|\nabla\vecu_n(t)\|^2_{\mL^2} 
&=
\bigl(
-\inpro{\Delta\vecu_n}{F_n(\vecu_n)}_{\mL^2} 
+
\frac12 \sum_{k=1}^n \|\nabla G_{nk}(\vecu_n)\|^2_{\mL^2}
\bigr)\dt\nn\\
&\quad-
\sum_{k=1}^n \inpro{\Delta\vecu_n}{G_{nk}(\vecu_n)}_{\mL^2}\,dW_k(t).
\end{align}
Using~\eqref{eq: adj2}, we infer that
\begin{align}
-\inpro{\Delta\vecu_n}{F_n(\vecu_n)}_{\mL^2} 
&=
-\kappa_1\|\Delta \vecu_n(t)\|^2_{\mL^2}
+\kappa_2\inpro{\Delta\vecu_n}{\bigl(1+ \mu|\vecu_n|^2\bigr)\vecu_n}_{\mL^2}\nn\\
&\quad-\frac12 \sum_{k=1}^n \gamma
\inpro{\Delta\vecu_n}{G_{nk}(\vecu_n)\times \vech_k}_{\mL^2}\nn \\
&=
-\kappa_1\|\nabla \vecu_n(t)\|^2_{\mL^2}
-\kappa_2\int_D \bigl(1+ \mu|\vecu_n|^2\bigr)|\nabla\vecu_n|^2\dvx\nn\\
&\quad
-2\mu\kappa_2\bigl(\inpro{\vecu_n}{\nabla\vecu_n}_{\mL^2}\bigr)^2
-\frac12\sum_{k=1}^n \|\nabla G_{nk}(\vecu_n)\|^2_{\mL^2}\nn\\
&\quad
+\sum_{k=1}^n 
R(\vecu_n,\vech_k),\label{eq: ito22}\\
\text{and }\quad
-\inpro{\Delta\vecu_n}{G_{nk}(\vecu_n)}_{\mL^2}
&=
\inpro{\nabla\vecu_n}{\gamma\vecu_n\times\nabla\vech_k+\kappa_1\nabla\vech_k}_{\mL^2},\label{eq: ito23}
\end{align}
where,
\begin{equation}\label{eq: defR}
R(\vecu_n,\vech_k)
=
 \frac12\gamma
\inpro{\nabla\vecu_n}{G_{nk}(\vecu_n)\times\nabla\vech_k}_{\mL^2}
+
\frac12
\inpro{\gamma\vecu_n\times\nabla\vech_k + \kappa_1\nabla\vech_k}{\nabla G_{nk}(\vecu_n)}_{\mL^2}.
\end{equation}
Therefore, it follows from~\eqref{eq: ito21}--\eqref{eq: ito23} that 
\begin{align}\label{eq: BIV0}
\frac12 \|\nabla\vecu_n(t)\|^2_{\mL^2} 
&+\kappa_1\int_0^t\|\Delta \vecu_n(s)\|^2_{\mL^2}\ds
+\kappa_2\int_0^t\int_D \bigl(1+ \mu|\vecu_n(s,\vecx)|^2\bigr)|\nabla\vecu_n(s,\vecx)|^2\dvx\ds\nn\\
&+2\mu\kappa_2\int_0^t \bigl(\inpro{\vecu_n(s)}{\nabla\vecu_n(s)}_{\mL^2}\bigr)^2\ds\nn\\
&=
\frac12 \|\nabla\vecu_{0,n}\|^2_{\mL^2} 
+
\sum_{k=1}^n \int_0^t
R(\vecu_n(s),\vech_k)\ds\nn\\
&\quad+
\sum_{k=1}^n \int_0^t \inpro{\nabla\vecu_n(s)}{\gamma\vecu_n(s)\times\nabla\vech_k+\kappa_1\nabla\vech_k}_{\mL^2}\,dW_k(s).
\end{align}
Using the Jensen inequality we deduce from~\eqref{eq: BIV0} that 
\begin{align}\label{eq: ito24}
\|\nabla\vecu_n(t)\|^{2p}_{\mL^2} 
&+\bigl(\int_0^t\|\Delta \vecu_n(s)\|^2_{\mL^2}\ds\bigr)^p
+\bigl(\int_0^t\int_D |\vecu_n(s,\vecx)|^2|\nabla\vecu_n(s,\vecx)|^2\dvx\ds\bigr)^p
\nn\\
&\leq 
c
+
c\bigl(\sum_{k=1}^n \int_0^t
\bigl|R(\vecu_n(s),\vech_k)\bigr|\ds\bigr)^p\nn\\
&\bigskip\quad\quad+
c\bigl(\sum_{k=1}^n \int_0^t \inpro{\nabla\vecu_n(s)}{\gamma\vecu_n(s)\times\nabla\vech_k+\kappa_1\nabla\vech_k}_{\mL^2}\,dW_k(s)\bigr)^p.
\end{align}
We now first estimate $\sum_{k=1}^n R(\vecu_n,\vech_k)$ by using H\"older inequality, the assumption~\eqref{eq: condh} 
and Cauchy--Schwarz inequality
\begin{align}\label{eq: BIV01}
\sum_{k=1}^n \bigl|R(\vecu_n,\vech_k)\bigr|
&\leq 
c\sum_{k=1}^n\bigl(\|\nabla\vecu_n\|_{\mL^2}\|G_{n,k}(\vecu_n)\|_{\mL^2}
+
 \bigl(\|\vecu_n\|_{\mL^2}+\|\nabla\vech_k\|_{\mL^2}\bigr)\|\nabla G_{n,k}(\vecu_n)\|_{\mL^2}\bigr)\nn\\
&\leq
c\|\vecu_n\|_{\mH^1}^2 
+ c
\sum_{k=1}^n\|G_{n,k}(\vecu_n)\|_{\mH^1}^2
+
c\sum_{k=1}^n \|\nabla\vech_k\|_{\mL^2}^2\nn\\
&\leq 
c + c\|\vecu_n\|_{\mH^1}^2. 
\end{align}
This inequality together with Lemma~\ref{lem: appSo_sta1} yields
\begin{align}\label{eq: ito25}
\mE\bigl[\bigl(\sum_{k=1}^n \int_0^t\bigl|R(\vecu_n(s),\vech_k)\bigr|\ds\bigr)^p\bigr]
\leq 
c + c\mE\bigl[\bigl(\int_0^t\|\vecu_n(s)\|_{\mH^1}^2\ds\bigr)^p\bigr]
\leq c. 
\end{align}
Then by using the Burkholder-Davis-Gundy inequality,~\eqref{eq: condh} and H\"older inequality, 
we estimate the last term in the right hand side of~\eqref{eq: ito24}
\begin{align}\label{eq: ito26}
\mE \sup_{s\in[0,t]}\bigl|
&\sum_{k=1}^n 
\int_0^s 
\inpro{\nabla\vecu_n(\tau)}{\gamma\vecu_n(\tau)\times\nabla\vech_k+\kappa_1\nabla\vech_k}_{\mL^2}\,dW_k(\tau)
\bigr|^p\nn\\
&\leq 
c\mE
\bigl|
\sum_{k=1}^n \int_0^t\bigl(
\inpro{\nabla\vecu_n(s)}{\gamma\vecu_n(s)\times\nabla\vech_k+\kappa_1\nabla\vech_k}_{\mL^2}\bigr)^2\ds
\bigr|^{p/2}\nn\\
&\leq 
c \mE
\bigl|\sum_{k=1}^n
\int_0^t\bigl(\|\nabla\vech_k\|_{\mL^\infty} \int_D
|\nabla\vecu_n(s,\vecx)|\,|\vecu_n(s,\vecx)|\dvx+\|\nabla\vecu_n(s)\|_{\mL^2}\|\nabla\vech_k\|_{\mL^2}
\bigr)^2\ds
\bigr|^{p/2}\nn\\
&\leq 
c\mE
\bigl|
\int_0^t\bigl(\int_D|\nabla\vecu_n(s,\vecx)|^2\,|\vecu_n(s,\vecx)|^2\dvx
+
c\|\nabla\vecu_n(s)\|_{\mL^2}^2
\bigr)\ds\bigr|^{p/2}\nn\\
&\leq 
c\mE
\bigl|
\int_0^t\int_D|\nabla\vecu_n(s,\vecx)|^2\,|\vecu_n(s,\vecx)|^2\dvx\ds\bigr|^{p/2}
+
c\mE
\bigl|\int_0^t \|\nabla\vecu_n(s)\|_{\mL^2}^2\ds\bigr|^{p/2}\nn\\
&\leq 
\mE\bigl[
\frac12\bigl(\int_0^t\int_D|\nabla\vecu_n(s,\vecx)|^2\,|\vecu_n(s,\vecx)|^2\dvx\ds\bigr)^p
+ \frac{c^2}{2}
\bigr]
+
c\mE
\int_0^t \|\nabla\vecu_n(s)\|_{\mL^2}^{2p}\ds
+c\nn\\
&\leq 
c+
\frac12
\mE\bigl(
\int_0^t\int_D|\nabla\vecu_n(s,\vecx)|^2\,|\vecu_n(s,\vecx)|^2\dvx\ds\bigr)^p
+
c\mE
\int_0^t \|\nabla\vecu_n(s)\|_{\mL^2}^{2p}\ds.
\end{align}
It follows from~\eqref{eq: ito24}--\eqref{eq: ito26} that
\begin{align*}
\mE\|\nabla\vecu_n(t)\|^{2p}_{\mL^2} 
+\mE\bigl(\int_0^t\|\Delta \vecu_n(s)\|^2_{\mL^2}\ds\bigr)^p
&+\frac12\mE\bigl(\int_0^t\int_D|\vecu_n(s,\vecx)|^2|\nabla\vecu_n(s,\vecx)|^2\dvx\ds\bigr)^p\nn\\
&\leq 
c
+
c\mE
\int_0^t \|\nabla\vecu_n(s)\|_{\mL^2}^{2p}\ds\,.
\end{align*}
The result follows immediately by using Gronwall's inequality, which complete the 
proof of this lemma.
\end{proof}
\begin{lemma}\label{lem: appSo_sta3}
For  every $t\in[0,T]$, there holds 
\begin{equation*}
\mE\int_0^t\|(1+\mu|\vecu_n(s)|^2)\vecu_n(s)\|^2_{\mL^2}\ds
\leq c.
\end{equation*}
Here  $c$ is a positive constant depending on $T$, $C_1$ and $h$.
\end{lemma}
\begin{proof}

From Lemma~\ref{lem: appSo_sta1}--\ref{lem: appSo_sta2} and the Sobolev imbedding of $\mH^1$ into $\mL^6$, we have
\begin{equation*}
\mE\bigl[\|\vecu_n^3(t)\|^2_{\mL^2}\bigr]
=
\mE\bigl[\|\vecu_n(t)\|^6_{\mL^6}\bigr]
\leq
\mE\bigl[\|\vecu_n(t)\|^6_{\mH^1}\bigr]
\leq
c,
\end{equation*}
so
\begin{align*}
\mE\bigl[\int_0^t\|\bigl(1+\mu|\vecu_n|^2(s)\bigr)\vecu_n(s)\|^2_{\mL^2}\ds\bigr]
&\leq
2\int_0^t\mE\bigl[\|\vecu_n(t)\|^2_{\mL^2}\bigr]\ds
+
2\mu^2\int_0^t\mE\bigl[\|\vecu_n^3(t)\|^2_{\mL^2}\bigr]\ds\\
&\leq
c,
\end{align*}
which completes the proof of the lemma.
\end{proof}
\begin{lemma}\label{lem: appSo_sta4}
For $p\geq 1 $, $n=1,2,\dots$, 
\begin{equation}\label{eq: uninf}
\mE\left(\int_0^T\|\vecu_n(s)\|^2_{\mL^{\infty}} \ds\right)^p
\leq c,
\end{equation}


Furthermore, for any $r\in\left[1,\frac{4}{3}\right)$ and $p\in[1,\infty)$ 
\begin{equation}\label{xxx}
\mathbb E\left(\int_0^T\left\|u_n(s)\times\Delta u_n(s)\right\|^r_{\mL^2}\,ds\right)^p\le c\,,
\end{equation}
where  $c$ is a positive constant depending on $T$, $C_1$ and $h$.
\end{lemma}
\begin{proof}
Estimate \eqref{eq: uninf} follows immediately from Lemmas~\ref{lem: appSo_sta1} and~\ref{lem: appSo_sta2} and the continuous imbedding $\mH^2\subset\mL^\infty$. 

\noindent 
We will prove \eqref{xxx} for $d=3$.  Using interpolation we obtain for every $t$ (omitted for simplicity) 
\begin{align}\label{eq: ex1}
\left\|u_n\times\Delta u_n\right\|_{\mL^2}&\le\left\|u_n\right\|_{\mL^\infty}\left\|\Delta u_n\right\|_{\mL^2}\nn\\
&\le \left\|u_n\right\|_{\mathbb H^{3/2}}\left\|\Delta u_n\right\|_{\mL^2}\nn\\
&\le \left\|u_n\right\|_{\mathbb H^{1}}^{1/2}\left\|u_n\right\|_{\mathbb H^2}^{1/2}\left\|u_n\right\|_{\mathbb H^2}\nn\\
&\le \left\|u_n\right\|_{\mathbb H^{1}}^{1/2}\left\|u_n\right\|_{\mathbb H^2}^{3/2}\,.
\end{align}
Therefore, using the H\"older inequality we obtain 
\[\begin{aligned}
\int_0^T\left\|u_n\times\Delta u_n\right\|_{\mL^2}^r\,dt&\le\int_0^T\left\|u_n\right\|_{\mathbb H^{1}}^{r/2}\left\|u_n\right\|_{\mathbb H^2}^{3r/2}\\
&\le\left(\int_0^T \left\|u_n\right\|_{\mathbb H^2}^{2}\,dt\right)^{3r/4}\left(\int_0^T \left\|u_n\right\|_{\mathbb H^{1}}^{2r/(4-3r)}\,dt\right)^{(4-3r)/4}
\end{aligned}\]
Using the H\"older inequality again and invoking Lemmas \ref{lem: appSo_sta1} and \ref{lem: appSo_sta2}  we find that 
\[\begin{aligned}
\mathbb E\int_0^T\left\|u_n\times\Delta u_n\right\|_{\mL^2}^r\,dt&\le\left(\mathbb E\int_0^T \left\|u_n\right\|_{\mathbb H^2}^{2}\,dt\right)^{3r/4}
\left(\mathbb E \int_0^T \left\|u_n\right\|_{\mathbb H^{1}}^{2r/(4-3r)}\,dt\right)^{(4-3r)/4}\\
&\le c\,,
\end{aligned}\]
for a certain $c>0$ independent of $n$ and \eqref{xxx} follows for $p=1$. Estimate for arbitrary $p>1$ follows easily by similar arguments as in the proof of Lemma \ref{lem: appSo_sta2}. 
\end{proof}
\section{Tightness and construction of new probability space and processes}\label{sec: tight}
Equation~\eqref{eq: GaLLB2} can be written in the following way as an approximation of equation~\eqref{eq: weakLLB}
\begin{align*}
\vecu_n(t) 
&= 
\vecu_n(0)
+\int_0^tF_n^1(\vecu_n(s))-
F_n^3(\vecu_n(s))\ds
+\int_0^tF_n^2(\vecu_n(s))\ds\nn\\
&\quad+
\frac12\sum_{k=1}^n \int_0^t\Pi_n\bigl(G_{nk}(\vecu_n(s))\times\vech_k\bigr)\ds
+\sum_{k=1}^n \int_0^tG_{nk}(\vecu_n(s))dW_k(s)\nn\\
\end{align*}
We will write shortly
\begin{equation}\label{eq_bn}
\vecu_n(t)=\vecu_n(0)+ \sum_{i=1}^3\vecB_{n,i}(\vecu_n)(t)+\vecB_{n,4}(\vecu_n,W)(t),\quad t\in[0,T].
\end{equation}
We now prove a uniform bound for $\vecu_n$.
\begin{lemma}\label{lem: unbound1}
Let $D\subset \R^d$ be an open bounded domain. 
Let $r\in\left[1,\frac43\right)$, $q\in[1,\infty)$, $p>2$ and $\alpha\in\left(0,\frac12\right)$ with $p\alpha>1$. 
Then there exists a constant $c$ depending on $p$, $C_1$ and $h$, such that for all $n\ge 1$
\begin{align}
\mE\|\vecB_{n,1}(\vecu_n)\|_{W^{1,2}(0,T;\mL^2)}^q
&\leq c,\label{eq: unbound1}\\
\mE\|\vecB_{n,2}(\vecu_n)\|^q_{W^{1,r}(0,T;\mL^2)}
&\leq c,\label{eq: unbound2}\\
\mE\|\vecB_{n,3}(\vecu_n)\|_{W^{1,2}(0,T;\mL^2)}^q 
&\leq c,\label{eq: unbound3}\\
\mE\|\vecB_{n,4}(\vecu_n,W)\|_{W^{\alpha,p}(0,T;\mL^2)}^q &\leq c,\label{eq: unbound4}
\end{align}
\textcolor{red}{Moreover, 
\begin{equation}\label{eq: unbound5}
\mE\|\vecu_n\|^q_{W^{\alpha,p}(0,T;\mL^2)}\leq c\,.
\end{equation}}
\end{lemma}
\begin{proof}
Inequality~\eqref{eq: unbound1} follows immediately from Lemma \ref{lem: appSo_sta1},  \ref{lem: appSo_sta2} and Lemma  \ref{lem: appSo_sta3}. Inequality ~\eqref{eq: unbound2} is in fact a reformulation of \eqref{xxx}. Inequality \eqref{eq: unbound3} follows from Lemma \ref{lem: appSo_sta1}. 
Estimate \eqref{eq: unbound4} is a consequence of Lemma \ref{lem: appSo_sta1} and Lemma~\ref{lem: appendix1}. In order to prove \eqref{eq: unbound5}, we recall the Sobolev embedding 
\[W^{1,r}(0,T;\mL^2)\subset W^{\alpha,p}(0,T;\mL^2),\quad\mathrm{if}\quad \frac1p-\alpha>\frac1r-1\,.\]
Therefore, using the first four inequalities, we easily deduce \eqref{eq: unbound5}. 
\end{proof}
\begin{lemma}\label{lem: tight}
If $\beta>0$ and $p\in\left(1,\frac43\right)$, then the measures $\{\cL(\vecu_n)\}_{n\in\N}$ on 
$L^2(0,T;\mH^1)\cap L^p(0,T;\mL^4)\cap C\bigl([0,T];X^{-\beta}\bigr)$ are tight.
\end{lemma}
\begin{proof}
From Lemmas~\ref{lem: appSo_sta1}--~\ref{lem: appSo_sta2} and~\eqref{eq: unbound5}, we deduce 
\begin{align*}
\mE\|\vecu_n\|_{W^{\alpha,p}(0,T;\mathbb L^2)\cap L^p(0,T;\mH^1)\cap L^2(0,T;\mH^2) }
&\leq c.
\end{align*}
This together with the following compact embeddings 
\begin{align}
W^{\alpha,p}(0,T;\mathbb L^2)\cap L^p(0,T;\mH^1)
\hookrightarrow C([0,T];X^{-\beta})\cap L^p(0,T;\mL^4),\label{eq:comp1}\\
W^{\alpha,p}(0,T;\mL^2)\cap L^2(0,T;\mH^2)
\hookrightarrow
C([0,T];\mL^2)\cap L^2(0,T;\mH^1)\label{eq:comp2}
\end{align}
imply the tightness of $\{\cL(\vecu_n)\}_{n\in\N}$.
\end{proof}
By Lemma~\ref{lem: tight} and the Prokhorov theorem, we have the following property by noting that 
from  the Kuratowski theorem, the Borel subsets of $C([0,T];S_n)$ are Borel subsets of $L^p(0,T;\mL^4)\cap C([0,T];X^{-\beta})\cap L^2(0,T;\mH^1)$.
\begin{proposition}\label{pro: newProcess}
Assume that $\beta>0$ and $p>1$. Then there exist 
\begin{enumerate}
 \item a propability space $(\Omega',\cF',\mP')$,
 \item a sequence $\{(\vecu_n',W_n')\}$ of random variables defined on $(\Omega',\cF',\mP')$ and taking 
 values in  the space $\bigl(L^p(0,T;\mL^4)\cap C([0,T];X^{-\beta})\cap L^2(0,T;\mH^1)\bigr)\times C([0,T];\R^\infty )$,
 \item a random variable $(\vecu',W')$ defined on $(\Omega',\cF',\mP')$ and taking values in 
 $\bigl(L^p(0,T;\mL^4)\cap C([0,T];X^{-\beta})\cap L^2(0,T;\mH^1)\bigr)\times C([0,T];\R)$,
\end{enumerate}
such that in the space $\bigl(L^p(0,T;\mL^4)\cap C([0,T];X^{-\beta})\cap L^2(0,T;\mH^1)\bigr)\times C([0,T];\R^\infty )$ there hold
\begin{enumerate}[label = (\alph*)]
 \item $\cL(\vecu_n,W) = \cL(\vecu_n',W_n') $, $n\in \N$,
 \item $(\vecu_n',W_n')\goto (\vecu',W')$ strongly, $\mP'$-a.s..
\end{enumerate}

Moreover, for every $p\in[1,\infty)$ the sequence $\{\vecu_n'\}_{n\in \N}$ satisfies
\begin{align}
\sup_{n\in \N}\mE'\bigg[\sup_{t\in[0,T]}\|\vecu'_n(t)\|_{\mH^1}^{2p} + \|\vecu'_n\|_{L^2(0,T;\mH^2)}^{2p}\bigg] 
&< \infty,\label{eq: un'bound1}\\
\sup_{n\in \N}\mE'\bigg[\int_0^T\|(1+\mu|\vecu_n'|^2(s))\vecu_n'(s)\|^2_{\mL^2}\ds\bigg]
&< \infty,\label{eq: un'bound2}
\end{align}
and for any $r\in\left[1,\frac43\right)$ and $p\in[1,\infty)$ 
\begin{align}
 \sup_{n\in \N}\mE' \left(\int_0^T\|\vecu_n'(s)\times\Delta\vecu_n'(s)\|^{r}_{\mL^2} \ds\right)^p
&< \infty,\label{eq: un'bound3}\\
\end{align}
\end{proposition}
It follows that $\vecu_n'\in C([0,T];S_n)$ and the laws on $C([0,T];S_n)$ of $\vecu_n$ and $\vecu_n'$ are equal.
\section{Existence of a weak solution}\label{sec: exist}
Our aim is to prove that $\vecu'$ from Proposition~\ref{pro: newProcess} is a weak solution of the stochastic LLBEs 
according to the Definition~\ref{def: weakso}.
We first find an equation satisfied by the new process $(\vecu_n'(t),W_n'(t))_{t\in[0,T]}$ in Subsection 5.1. 
Then in Subsection 5.2 
we  prove the convergence of that equation.

\subsection{Equation for the new process}
The following lemmas state that the processes $W'$ and $W_n'$ from Proposition~\ref{pro: newProcess} 
are Brownian motions, which can be proved as in~\cite{BrzGolJer12}.
\begin{lemma}\label{lem: Brownian1}
The processes $W_n'$, $n\ge 1$, and $W^\prime$ are Wiener processes defined on $(\Omega',\cF',\mP')$. Moreover, for $0\leq s<t\leq T$, the increments $W'(t)-W'(s)$ are independent of the $\sigma$-algebra generated by 
$\vecu'(r)$ and $W'(r)$ for $r\in [0,s]$.
\end{lemma}
\par\bigskip
From now on, we work solely in the probability space $\left(\Omega',\mathcal F',\mF',\mP'\right)$ and all the processes are defind on this space. In order to simplify notations, we will write $\left(\Omega,\mathcal F,\mF,\mP\right)$ and the new processes $W_n^\prime,\,\vecu_n^\prime$ etc. will be denoted as $W_n,\vecu_n\ldots$ etc. 
\begin{lemma}\label{lem: eqNewPro}
Let $\vecB_{n,i}$ be defined as in \eqref{eq_bn}. Let a sequence of $\mL^2$-valued processes $(\vecM_n(t))_{t\in[0,T]}$  on 
$(\Omega, \cF,\mP)$ be defined by
\begin{align*}
\vecM_n(t)
:=
\vecu_n(t) - \vecu_n(0) - \sum_{i=1}^3\vecB_{n,i}(\vecu_n)(t).
\end{align*}
Then for each $ t\in[0,T]$ there holds
\[
\vecM_n(t) = \vecB_{n,4}(\vecu_n,W_n)(t)\quad\mP\text{-a.s. }\] 
\end{lemma}
\begin{proof}
The result is obtained by using~\eqref{eq: un'bound1}, Lemma~\ref{lem: appSo_sta1} and the same arguments as in~\cite[Theorem 7.7 (Step 1)]{ZdzisLiang2014}.
\end{proof}
\subsection{Convergence of the new processes}
Before proving the convergence of  $\{\vecM_n\}$, we find the limits of sequences $\{\vecB_{n,i}(\vecu_n)\}$ for $i=1,2,3$,  
and their relationship with $\vecu'$ in the following lemmas.

\begin{lemma}\label{lem: u'conve4}
For any $\vecphi\in \mL^4(D)\cap X^{\beta}$, there holds
\begin{align}
\lim_{n\goto \infty}\mE\int_0^t
 \iprod{\Pi_n\bigl[(1+\mu|\vecu_n|^2)\vecu_n(s)\bigr],\vecphi}_{\mL^2} \ds
&= \mE\int_0^t
\iprod{(1+\mu|\vecu|^2)\vecu(s),\vecphi}_{\mL^2}\ds,\label{eq: u'conve4}\\
\lim_{n\goto \infty}\mE\int_0^t
\iprod{\Pi_n\bigl((\vecu_n(s)\times\vech_k)\times\vech_k\bigr),\vecphi}_{\mL^2} \ds
&= \mE\int_0^t
\iprod{(\vecu(s)\times\vech_k)\times\vech_k,\vecphi}_{\mL^2} \ds.\label{eq: u'conve6}
\end{align}
\end{lemma}
\begin{proof}

\underline{Proof of~\eqref{eq: u'conve4}:}
By using the same arguments in the proof of~\cite[Lemma 4.3]{Le2016} we have $\mP$-a.s.
\begin{equation}\label{eq: u'conve41}
\lim_{n\goto\infty}
\int_0^t
\iprod{\Pi_n\bigl((1+\mu|\vecu_n|^2(s))\vecu_n(s)\bigr),\vecphi}_{\mL^2}\ds
=
\int_0^t
\iprod{(1+\mu|\vecu|^2(s))\vecu(s),\vecphi}_{\mL^2}\ds.
\end{equation}
We have 
\begin{align*}
\mE\bigl|
\iprod{\Pi_n\bigl((1+\mu|\vecu_n|^2)\vecu_n\bigr),\vecphi}_{L^2(0,t;\mL^2)}\bigr|^2
&\leq 
2\mE\bigl|
\iprod{\vecu_n,\vecphi}_{L^2(0,t;\mL^2)}\bigr|^2 \\
&\quad+
2\mE\bigl|
\iprod{\mu|\vecu_n|^2\vecu_n,\Pi_n\vecphi}_{L^2(0,t;\mL^2)}\bigr|^2\\
&\leq 
2\bigl(\mE\|\vecu_n\|_{L^2(0,T;\mL^2)}^4\bigr)^{\frac12}
\bigl(\mE\|\vecphi\|_{L^2(0,T;\mL^2)}^4\bigr)^{\frac12}\\
&\quad+
2\mu\bigl(\mE\|\vecu_n\|_{L^6(0,T;\mL^6)}^{12}\bigr)^{\frac12}
\bigl(\mE\|\Pi_n\vecphi\|_{L^2(0,T;\mL^2)}^4\bigr)^{\frac12}.
\end{align*}
This together with~\eqref{eq: boundPi_n},~\eqref{eq: un'bound1} and the Sobolev imbedding 
of $\mH^1$ into $\mL^6$ imply
\begin{equation}\label{eq: u'conve42}
\sup_{n\in\N}
\mE\bigl|
\iprod{\Pi_n\bigl((1+\mu|\vecu_n|^2)\vecu_n\bigr),\vecphi}_{L^2(0,t;\mL^2)}\bigr|^2<\infty.
\end{equation}
From~\eqref{eq: u'conve41} and~\eqref{eq: u'conve42}, the first result~\eqref{eq: u'conve4} 
follows immediately by using the Vitali theorem.

\underline{Proof of~\eqref{eq: u'conve6}:}
The proof of~\eqref{eq: u'conve6} is omitted because it is similar to the proof of~\eqref{eq: u'conve4}.
\end{proof}
\begin{lemma}\label{lem: un'conve7}
Let $\{\vecu_n\}$ and $\vecu$ be the processes defined in Proposition~\ref{pro: newProcess}. Then 
for any $p\geq 1$ there hold
\[
\vecu_n\goto\vecu\text { weakly in } L^{2p}(\Omega;\bigl(L^{\infty}(0,T;\mH^1)\cap L^2(0,T;\mH^2)\bigr)),
\]
hence $\vecu'\in L^{2p}\bigl(\Omega; C([0,T];\mH^1_w)\bigr)$ and
\[
 \mE\sup_{t\in[0,T]}\|\vecu(t)\|_{\mH^1}^{2p} + \mE\left(\int_0^T\|\vecu(t)\|_{\mH^2}^{2}\,dt\right)^p
 < \infty. 
\]
\end{lemma}
\begin{proof}
From $(b)$-Proposition~\ref{pro: newProcess}, we get for
$\omega\in\Omega'$, $\mP'$-a.s.
\begin{align*}
\int_0^T\,
 _{\mL^4}\!\inpro{\vecu_n(t,\omega)}{\vecpsi(t,\omega)}_{\mL^{\frac43}}\dt
\goto 
\int_0^T\,
_{\mL^4}\!\inpro{\vecu(t,\omega)}{\vecpsi(t,\omega)}_{\mL^{\frac43}}\dt,
\end{align*}
for $\vecpsi(\omega)\in L^{\frac43}(0,T;\mL^{\frac43})$.
Moreover, the sequence $\int_0^T\,
_{\mL^4}\!\inpro{\vecu_n(t)}{\vecpsi(t)}_{\mL^{\frac43}}\dt$ 
is uniformly integrable on $\Omega$ if we choose $\vecpsi\in L^4(\Omega; L^{\frac43}(0,T;\mL^{\frac43}))$. 
Indeed, 
\begin{align*}
\sup_{n\in \N}\int_{\Omega}
\bigl|
\int_0^T\,
_{\mL^4}\!\inpro{\vecu_n(t)}{\vecpsi(t)}_{\mL^{\frac43}}\dt
\bigr|^2\,d\mP(\omega)
&\leq 
\sup_{n\in \N}\int_{\Omega}
\bigl|
\int_0^T
\|\vecu_n(t)\|_{\mL^4}
\|\vecpsi(t)\|_{\mL^{\frac43}}\dt
\bigr|^2\,d\mP'(\omega)\\
&\leq 
\sup_{n\in \N}\int_{\Omega}
\|\vecu_n\|_{L^{\infty}(0,T;\mL^4)}^2
\|\vecpsi\|_{L^1(0,T;\mL^{\frac43})}^2\,d\mP(\omega)\\
&\leq 
\sup_{n\in \N}
\|\vecu_n\|_{L^4(\Omega;L^{\infty}(0,T;\mL^4))}^2
\|\vecpsi\|_{L^4(\Omega;L^1(0,T;\mL^{\frac43}))}^2\\
&<\infty,
\end{align*}
here the last inequality is obtained by using~\eqref{eq: un'bound1} and the imbedding of $\mH^1$ into $\mL^4$. 
Thus, by using the Vitali theorem we deduce
\begin{align}\label{eq: un'conve72}
\mathbb E\int_0^T\,
_{\mL^4}\!\inpro{\vecu_n(t)}{\vecpsi(t)}_{\mL^{\frac43}}\dt
\goto 
\mathbb E
\int_0^T\,
_{\mL^4}\!\inpro{\vecu(t)}{\vecpsi(t)}_{\mL^{\frac43}}\dt\,.
\end{align}
On the other hand, by using the Banach-Alaoglu theorem we infer from~\eqref{eq: un'bound1} that 
there exist a subsequence of $\{\vecu_n\}$ (still denoted  by $\{\vecu_n\}$) 
and $\vecv\in L^{2p}(\Omega;L^{\infty}(0,T;\mH^1)\cap L^2(0,T;\mH^2))$ such that 
\[
\vecu_n\goto\vecv\text { weakly in } L^{2p}(\Omega;L^{\infty}(0,T;\mH^1)\cap L^2(0,T;\mH^2)).
\]
In particular, since $L^{2p}(\Omega;L^{\infty}(0,T;\mH^1)\cap L^2(0,T;\mH^2))$ is isomorphic to the space 
$\bigl(L^{\frac{2p}{2p-1}}(\Omega;L^1(0,T;X^{-\frac12})\cap L^2(0,T;X^{-1}))\bigr)^*$, we have
\begin{equation}\label{eq: un'conve71}
\mathbb E\int_0^T\int_D\vecu_n(t,\vecx)\vecphi(t,\vecx)\dvx\dt
 \goto
\mathbb E\int_0^T\int_D\vecv(t,\vecx)\vecphi(t,\vecx)\dvx\dt,
\end{equation}
as $n$ tends to infinity, for any $\vecphi\in L^{\frac{2p}{2p-1}}(\Omega;L^1(0,T;X^{-\frac12})\cap L^2(0,T;X^{-1}))$.

By the density of $L^4(\Omega;L^2(0,T;\mL^{\frac43}))$ in 
$L^{\frac{2p}{2p-1}}(\Omega;L^1(0,T;X^{-\frac12})\cap L^2(0,T;X^{-1}))$, we infer from~\eqref{eq: un'conve72} and~\eqref{eq: un'conve71} 
that $\vecu=\vecv$ in $L^{2p}(\Omega;L^{\infty}(0,T;\mH^1)\cap L^2(0,T;\mH^2))$. 
It follows from Proposition~\eqref{pro: newProcess} that $\vecu_n\in C([0,T];\mH^1_w)$. This together with the weakly convergence 
of $\vecu_n$ to $\vecu$ in $L^{2p}(\Omega;\bigl(L^{\infty}(0,T;\mH^1))$ and the completeness of $C([0,T];\mH^1_w)$ 
imply that $\vecu\in  L^{2p}\bigl(\Omega;C([0,T];\mH^1_w)\bigr)$.

Furthermore, 
since $\vecv$ satisfies~\eqref{eq: un'bound1}, it implies  $\vecu$ also satisfies~\eqref{eq: un'bound1}, which 
completes the proof of the lemma.
\end{proof}

\noindent
Let $p\in\left(1,\frac43\right)$ be fixed. From~\eqref{xxx} and by the Banach-Alaoglu theorem, there  
exist subsequences of $\{\vecu_n\times\Delta\vecu_n\}$ and of $\{\Pi_n(\vecu_n\times\Delta\vecu_n)\}$ 
(still denoted by $\{\vecu_n\times\Delta\vecu_n\}$, $\{\Pi_n(\vecu_n\times\Delta\vecu_n)\}$, respectively); 
and $\vecZ, \bar{\vecZ}\in L^1(\Omega;L^{p}(0,T;\mL^2))$ such that 
\begin{align}
\vecu_n\times\Delta\vecu_n
&\goto \bar{\vecZ} \text{ weak$^\star$ in }L^p(\Omega;L^{p}(0,T;\mL^2))\label{con: Z}\\
\Pi_n(\vecu_n\times\Delta\vecu_n)
&\goto \vecZ \text{ weak$^\star$ in }L^p(\Omega;L^{p}(0,T;\mL^2)).\label{con: Zbar}
\end{align}
Using the same arguments as in~\cite[Lemma 4.2]{Le2016}, we obtain
\begin{equation}\label{eq: Z=Zbar}
\vecZ = \bar{\vecZ} \text{ in } L^p(\Omega;L^{p}(0,T;\mL^2)).
\end{equation}
\begin{lemma}\label{lem: u'conve5}
For any $\vecphi\in  L^{\infty}(\Omega;L^{4}(0,T;\mW^{1,4}))$, there holds
 \[
  \lim_{n\goto \infty}\mE\int_0^T
\iprod{\Pi_n\bigl(\vecu_n(s)\times\Delta\vecu_n(s)\bigr),\vecphi}_{\mL^2}\ds
= -\mE\int_0^T
\iprod{\vecu(s)\times\nabla\vecu(s),\nabla\vecphi}_{\mL^2}\ds\,.
 \]
\end{lemma}
\begin{proof}
From~\eqref{con: Z}--\eqref{eq: Z=Zbar} and 
\[
\iprod{\vecu_n(t)\times\Delta\vecu_n(t),\vecphi}_{\mL^2}=-\iprod{\vecu_n(t)\times\nabla\vecu_n(t),\nabla\vecphi}_{\mL^2}, 
\]
it is sufficient to prove that 
\begin{equation}\label{eq: conve51}
  \lim_{n\goto \infty}\mE\int_0^T
\iprod{\vecu_n(s)\times\nabla\vecu_n(s),\nabla\vecphi}_{\mL^2}\ds 
= \mE\int_0^T
\iprod{\vecu(s)\times\nabla\vecu(s),\nabla\vecphi}_{\mL^2}\ds\,.
\end{equation}
Using the same  arguments in the proof of~\cite[Lemma 4.3]{Le2016}, we have $\mP$-a.s.
\begin{equation*}
 \lim_{n\goto \infty}\int_0^T
\iprod{\vecu_n(t)\times\nabla\vecu_n(t),\nabla\vecphi(t)}_{\mL^2}\dt
= \int_0^T
\iprod{\vecu(t)\times\nabla\vecu(t),\nabla\vecphi(t)}_{\mL^2}\dt\,.
\end{equation*}
Moreover, the sequence $\int_0^T
\iprod{\vecu_n(t)\times\nabla\vecu_n(t),\nabla\vecphi(t)}_{\mL^2}\dt $ is uniformly 
integrable on $\Omega$. Indeed, ~\eqref{eq: un'bound1} and the Sobolev 
imbedding of $\mH^1$ into $\mL^4$ yield
\begin{align*}
\sup_{n\in\N}\mathbb E
&\bigl|
\int_0^T
\iprod{\vecu_n(t)\times\nabla\vecu_n(t),\nabla\vecphi(t)}_{\mL^2}\dt
\bigr|^2\\
&\leq 
\sup_{n\in\N}\mathbb E
\bigl|
\int_0^T
\|\vecu_n(t)\|_{\mL^4}^2
\|\nabla\vecu_n(t)\|_{\mL^2}
\|\nabla\vecphi(t)\|_{\mL^4}^2\dt
\bigr|^2\\
&\leq c 
\sup_{n\in\N}\mathbb E\left(
\|\vecu_n\|_{L^{\infty}(0,T;\mH^1)}^6
\|\nabla\vecphi\|_{L^2(0,T;\mL^4)}^2\right)\\
&\leq c
\|\nabla\vecphi\|_{L^4(\Omega;L^2(0,T;\mL^4))}^2
\sup_{n\in\N} \|\vecu_n\|_{L^{12}(\Omega;L^{\infty}(0,T;\mH^1))}^6
<\infty.
\end{align*}
Thus the Vitali theorem yields~\eqref{eq: conve51}, which completes the proof of the lemma.
\end{proof}
\noindent 
Lemma~\ref{lem: u'conve5} together with~\eqref{con: Zbar} yields 
for any test function $\vecphi\in  L^{\infty}(\Omega;L^{4}(0,T;\mW^{1,4}))$ there holds
\begin{align*}
\mathbb E\int_0^T
\inpro{\vecZ(t)}{\vecphi(t)}_{\mL^2}\dt
&= 
-
\mathbb E\int_0^T
\inpro{\vecu(t)\times\nabla\vecu(t)}{\nabla \vecphi(t)}_{\mL^2}\dt\\
&=
\mathbb E\int_0^T
\inpro{\vecu(t)\times\Delta \vecu(t)}{\vecphi(t)}_{\mL^2}\dt,
\end{align*}
where the last equality follows from 
$\vecu\times\Delta \vecu\in L^p(\Omega;L^r(0,T;\mL^2))$ for any $p\in(1,\infty)$ and $r\in\left[1,\frac43\right)$. 
Hence, we deduce 
\begin{equation}\label{no: 2}
\vecZ
=\vecu\times\Delta \vecu
\quad\text{in } L^1(\Omega;L^{r}(0,T;\mL^2)).
\end{equation}
%
The limits of $\{\vecM_n\}$ and $\{\vecB_{n,4}(\vecu_n,W_n)\}$ as $n$ tends to infinity are stated in the following lemmas.
\begin{lemma}\label{lem: newProConver1}
For each $t\in[0,T]$, the sequence of random variables $\vecM_n(t)$ is weakly convergent in $L^{\frac43}(\Omega;X^{-\beta})$ to a limit $\vecM$
that satisfies the following equation
\begin{align*}
\vecM(t)=
&\vecu(t) -\vecu(0)
-
\kappa_1\int_0^t \Delta\vecu(s)\ds 
+
\kappa_2 \int_0^t (1+\mu|\vecu|^2)\vecu(s)\ds \\
&-
\gamma \int_0^t \vecu(s)\times\Delta\vecu(s)\ds 
-
\frac12 \gamma \sum_{k=1}^\infty 
\int_0^t (\vecu(s)\times\vech_k)\times\vech_k\ds
\end{align*}
\end{lemma}
\begin{proof}
Let $t\in(0,T]$ and $\vecphi\in L^4(\Omega;X^{\beta})$. 
Since $\vecu_n$ converges to $\vecu$ in $C([0,T];X^{-\beta})$ $\mP$-a.s., we infer that
\begin{equation}\label{eq: u'conve1}
\lim_{n\goto \infty}\,
_{X^{-\beta}}\!\iprod{\vecu_n(t),\vecphi}_{X^{\beta}} 
=
_{X^{-\beta}}\!\iprod{\vecu(t),\vecphi}_{X^{\beta}}, 
\quad\text{$\mP'$-a.s.}.
\end{equation}
Furthermore, by using $\mH^1 \hookrightarrow X^{-\beta}$  and~\eqref{eq: un'bound1} we obtain
\begin{align*}
\sup_{n\in \N} \mE
\bigl|\,_{X^{-\beta}}\!\iprod{\vecu_n(t),\vecphi}_{X^{\beta}}\bigr|^2 
\leq
\sup_{n\in \N} 
\bigl(\mE\left[\|\vecu_n(t)\|_{X^{-\beta}}^4\right]\bigr)^{\frac12}
\bigl(\mE\left[\|\vecphi\|_{X^{\beta}}^4\right]\bigr)^{\frac12}
<\infty,
\end{align*}
which implies that $\{_{X^{-\beta}}\!\iprod{\vecu_n(t),\vecphi}_{X^{\beta}}\}_{n\in\N}$ is uniformly integrable. 
Together with~\eqref{eq: u'conve1}, it implies from the Vitali theorem that
\begin{equation}\label{eq: u'conve2}
 \lim_{n\goto \infty}\mE\,_{X^{-\beta}}\!\iprod{\vecu_n(t),\vecphi}_{X^{\beta}} 
= \mE\,_{X^{-\beta}}\!\iprod{\vecu(t),\vecphi}_{X^{\beta}}\,.
\end{equation}
By using Lemmas~\ref{lem: un'conve7} and~\ref{lem: u'conve5}, 
we infer from~\eqref{no: 2} and the embedding 
\[
L^1(\Omega;L^{r}(0,T;\mL^2)) 
\hookrightarrow L^1(\Omega;L^{r}(0,T;X^{-\beta})), 
\]
that
\begin{align}
 \lim_{n\goto \infty}\mE\int_0^t
 \iprod{\Delta\vecu_n(s),\vecphi}_{\mL^2} \ds
&= \mE\int_0^t
\iprod{\Delta\vecu(s),\vecphi}_{\mL^2}\ds\,,\label{eq: u'conve3}\\
\lim_{n\goto \infty}\mE\int_0^t\,
_{X^{-\beta}}\!\iprod{\Pi_n\bigl(\vecu_n(s)\times\Delta\vecu_n(s)\bigr),\vecphi}_{X^{\beta}}\ds 
&= \mE\int_0^t\,
_{X^{-\beta}}\!\iprod{\vecu(s)\times\Delta\vecu(s),\vecphi}_{X^{\beta}}\ds\,,\label{eq: u'conve5}
\end{align}
These limits together with Lemma~\ref{lem: u'conve4} imply that 
\begin{equation*}
\lim_{n\goto \infty}\,
_{L^{\frac43}(\Omega;X^{-\beta})}\!\iprod{\vecM_n(t),\vecphi}_{L^4(\Omega;X^{\beta})}
=\,
_{L^{\frac43}(\Omega;X^{-\beta})}\!\iprod{\vecM(t),\vecphi}_{L^4(\Omega;X^{\beta})},
\end{equation*}
which complete the proof of this Lemma.
\end{proof}
\begin{lemma}\label{lem: newProConver2}
Let $\{\vecu_n\}$ and $\vecu'$ be the processes defined in Proposition~\ref{pro: newProcess}. 
Then there holds
\begin{equation*}
\lim_{n\goto\infty}
\mE
\bigl\|
\sum_{k=1}^n \bigl(\int_0^t
\Pi_n(\gamma\vecu_n(s)\times\vech_k + \kappa_1\vech_k)\, dW_{k,n}(s)
-
\int_0^t
(\gamma\vecu(s)\times\vech_k + \kappa_1\vech_k)\, dW_{k}(s)
\bigr)
\bigr\|_{X^{-\beta}}=0\,.
\end{equation*}
\end{lemma}
\begin{proof}
The proof of this lemma is omitted because it is similar as part of the proof of~\cite[Lemma 5.2]{BrzGolJer12}.
\end{proof}
\noindent
\underline{\textbf{Proof of the main theorem (Theorem~\ref{theo: main}):}}
\begin{proof}
From Lemmas~\ref{lem: eqNewPro},~\ref{lem: newProConver1} and~\ref{lem: newProConver2} we deduce 
\[
 \vecM(t) 
 = \sum_{k=1}^\infty \int_0^t
(\gamma\vecu(s)\times\vech_k + \kappa_1\vech_k)\, dW_{k}(s)
\quad\text{in }L^{\frac43}(\Omega;X^{-\beta}),
\]
which means $\{\vecu,W\}$ satisfies~\eqref{eq: mainu}.\\
It remains to prove that $\vecu$ satisfies~\eqref{eq: mainu2}. 
Since $\vecu$ and $W$ satisfy~\eqref{eq: mainu} $\mP$-a.s., for $0\leq s<t\leq T$ we have 
\begin{align*}
\vecu'(t) -\vecu(\tau)
&=
\kappa_1\int_\tau^t \Delta\vecu(s)\ds 
+
\gamma \int_\tau^t \vecu(s)\times\Delta\vecu(s)\ds
-
\kappa_2 \int_\tau^t (1+\mu|\vecu|^2)\vecu'(s)\ds  \nn\\
&\quad+
\frac{\gamma}{2}
\sum_{k=1}^\infty 
\int_\tau^t
(\vecu(s)\times\vech_k)\times\vech_k\ds
+
\gamma
\sum_{k=1}^\infty 
\int_\tau^t
\vecu(s)\times\vech_k\, dW_k(s).
\end{align*}
By the Minkowski inequality and the the embedding $\mL^2\hookrightarrow\mL^{\frac32}$, we obtain for any $p\geq  1$
\begin{align}\label{eq: Holcont}
\mE\|\vecu(t) -\vecu(\tau)\|_{\mL^{3/2}}^{2p}
&\leq
c\mE\left(\int_\tau^t \|\Delta\vecu(s)\|_{\mL^{2}}\ds \right)^{2p}
+
c \mE\left(\int_\tau^t \|\vecu(s)\times\Delta\vecu(s)\|_{\mL^{3/2}}\ds\right)^{2p}\nn\\
&\quad
+
c \mE\left(\int_\tau^t \|\vecu(s)\|_{\mL^{2}}\ds \right)^{2p}
+
c \mE\left(\int_\tau^t \||\vecu|^2\vecu(s)\|_{\mL^{2}}\ds \right)^{2p}\nn\\
&\quad
+
c
\mE\left(
\int_\tau^t
\sum_{k=1}^\infty \|(\vecu(s)\times\vech_k)\times\vech_k\|_{\mL^{2}}\ds\right)^{2p}\nn\\
&\quad+
c\sum_{k=1}^\infty \mE\left(
\sum_{k=1}^\infty \|\int_\tau^t
\vecu(s)\times\vech_k\, dW_k(s)\|_{\mL^{2}}\right)^{2p}.
\end{align}
The following estimates follow from $\vecu\in L^{2p}\bigl(\Omega;\bigl(L^{\infty}(0,T;\mH^1)\cap L^2(0,T;\mH^2)\bigr)\bigr)$ 
and the embedding $\mH^1\hookrightarrow\mL^6$,
\begin{align*}
\mE\left(\int_\tau^t \|\Delta\vecu(s)\|_{\mL^{2}}\ds \right)^{2p}
&\leq 
(t-\tau)^p\mE\left(\int_0^T \|\Delta\vecu(s)\|_{\mL^{2}}^2\ds \right)^p
\leq 
c(t-\tau)^p;\\
\mE\left(\int_\tau^t \|\vecu(s)\times\Delta\vecu(s)\|_{\mL^{3/2}}\ds\right) ^{2p}
&\leq 
\mE\left(\int_\tau^t \|\vecu(s)\|_{\mL^{6}}\|\Delta\vecu(s)\|_{\mL^{2}}\ds\right) ^{2p}\\
&\leq 
c(t-\tau)^p
\mE\left(\|\vecu\|_{L^\infty(0,T;\mH^1)}^{2p}\bigl(\int_0^T \|\Delta\vecu(s)\|_{\mL^{2}}^2\ds\bigr)^p\right)\\
&\leq 
c(t-\tau)^p;\\
\mE\bigl[\int_\tau^t \|\vecu(s)\|_{\mL^{2}}\ds \bigr]^{2p}
&\leq 
(t-\tau)^{2p}\mE\bigl[\|\vecu\|_{L^{\infty}(0,T;\mL^2)}^{2p}\bigr]
\leq 
c(t-\tau)^{2p};\\
\mE\bigl[\int_\tau^t \||\vecu|^2\vecu(s)\|_{\mL^{2}}\ds \bigr]^{2p}
&\leq 
(t-\tau)^p\mE\bigl[\int_\tau^t\|\vecu(s)\|_{\mL^6}^6\ds \bigr]^{p}\\
&\leq 
(t-\tau)^{2p}\mE\bigl[\|\vecu\|_{L^\infty(0,T;\mH^1)}^{6p} \bigr]\\
&\leq 
c(t-\tau)^{2p};\\
\mE\bigl[\sum_{k=1}^\infty 
\int_\tau^t
\|(\vecu(s)\times\vech_k)\times\vech_k\|_{\mL^{2}}\ds\bigr]^{2p}
&\leq 
(t-\tau)^{2p}\bigl(\sum_{k=1}^\infty \|\vech_k\|_{\mL^\infty}^2\bigr)^{2p}\mE\bigl[\|\vecu\|_{L^{\infty}(0,T;\mL^2)}^{2p}\bigr]\\
&\leq 
c(t-\tau)^{2p};\\
\mE\bigl[\sum_{k=1}^\infty 
\|\int_\tau^t
\vecu(s)\times\vech_k\, dW_k(s)\|_{\mL^{2}}\bigr]^{2p}
&\leq 
c\mE\bigl[\sum_{k=1}^\infty 
\int_\tau^t
\|\vecu(s)\times\vech_k\|_{\mL^{2}}^2 \ds\bigr]^{p}\\
&\leq 
c(t-\tau)^p\bigl(\sum_{k=1}^\infty\|\vech_k\|_{\mL^\infty}^2\bigr)^{p}
\mE\bigl[
\|\vecu\|_{L^{\infty}(0,T;\mL^{2})}^{2p}\bigr]\\
&\leq 
c(t-\tau)^p.
\end{align*}
Here we use the Burkholder-Davis-Gundy inequality for the last estimate. These estimates together with~\eqref{eq: Holcont} yield
\begin{equation}\label{eq: ex2}
\mE\|\vecu(t) -\vecu(\tau)\|_{\mL^{3/2}}^{2p}\leq c(t-\tau)^p.
\end{equation}
Noting from~\eqref{eq: ex1} that
\begin{align*}
 \mE\left(\int_\tau^t \|\vecu(s)\times\Delta\vecu(s)\|_{\mL^2}\ds\right)^{2p}
 &\leq 
 \mE\left(\int_\tau^t \|\vecu(s)_{\mH^1}^{1/2}\|\vecu(s)\|_{\mH^2}^{3/2}\ds\right)^{2p}\\
 &\leq 
 (t-\tau)^{p/2}
 \mE\bigl[\|\vecu\|_{L^\infty(0,T;\mH^1)}^{p}\bigl(\int_0^T \|\vecu(s)\|_{\mH^2}^2\ds\bigr)^{3p/2}\bigr]\\
 &\leq c (t-\tau)^{p/2}.
\end{align*}
This together with~\eqref{eq: ex2}
imply that $\vecu$ satisfies~\eqref{eq: mainu2} (thanks to the Kolmogorov continuity test).
\end{proof}
\section{Existence of an invariant measure for the stochastic LLBE on 1 or 2-dimensional domains}\label{sec: invariant}
In this section we will show the existence of invariant measure for equation \eqref{eq: mainu}. In our proof we modify the ideas from \cite{motyl}, where different type of difficulties had to be dealt with. \\ 
We start with the following result. 
\begin{lemma}\label{lem: boundInProba}
Let $u$ be a weak solution to equation \eqref{eq: mainu} with properties listed in Theorem \ref{theo: main}. 
Then there exists a positive constant $c$ depending on $C_1$ and $h$ such that for all $t\ge 0$ we have
\[
\int_0^t\mE\|\vecu(s)\|^2_{\mH^2}\ds \leq c(1 + t).
\]
\end{lemma}
\begin{proof}
We will use a version of the It\^o Lemma proved in \cite{pardoux}. By Theorem \ref{theo: main} and with  $V=\mathbb H^1$ we easily find that assumptions of Lemma 1.4 in \cite{pardoux} are satsified and therefore  ~\eqref{eq: mainu} yields 
\begin{align}\label{eq: BIV1}
\frac12 d\|\vecu(t)\|^2_{\mL^2} 
&=
\bigl(
\inpro{\vecu(t)}{F(\vecu(t))}_{\mL^2} 
+
\frac12 \sum_{k=1}^\infty \|G_k(\vecu(t))\|^2_{\mL^2}
\bigr)\dt
+
\sum_{k=1}^\infty \inpro{\vecu(t)}{G_k(\vecu(t))}_{\mL^2}\,dW_k(t),
\end{align}
where 
\begin{align*}
 G_k(\vecu)&:= \gamma\vecu\times\vech_k + \kappa_1\vech_k,\nn\\
\text{and}\quad F(\vecu)&:= \kappa_1\Delta\vecu + \gamma\vecu\times\Delta\vecu - 
\kappa_2 (1+\mu|\vecu|^2)\vecu + \frac12 \sum_{k=1}^\infty G_k(\vecu)\times\vech_k.
\end{align*}
Noting
$\inpro{\vecu(t)}{G_{k}(\vecu(t))}_{\mL^2}
=
\kappa_1\inpro{\vecu(t)}{\vech_k}_{\mL^2}$ and
\begin{align*}
 \inpro{\vecu(t)}{F(\vecu(t))}_{\mL^2} 
&=
-\kappa_1\|\nabla \vecu(t)\|^2_{\mL^2}
-\kappa_2\int_D \bigl(1+ \mu|\vecu|^2\bigr)|\vecu|^2\dvx\nn\\
&\quad+\frac12 \sum_{k=1}^\infty \gamma
\inpro{\vecu}{G_{k}(\vecu)\times \vech_k}_{\mL^2}\nn \\
&=
-\kappa_1\|\nabla \vecu(t)\|^2_{\mL^2}
-\kappa_2\|\vecu(t)\|^2_{\mL^2}-\kappa_2\mu\|\vecu(t)\|_{\mL^4}^4\nn\\
&\quad
-\frac12\sum_{k=1}^\infty \|G_{k}(\vecu)\|^2_{\mL^2}
+\frac12\kappa_1 \sum_{k=1}^\infty \|\vech_k\|^2_{\mL^2},
\end{align*}
it follows from~\eqref{eq: BIV1} that 
\begin{align}\label{eq: BIV2}
 \frac12 \|\vecu(t)\|^2_{\mL^2} 
 &+\kappa_1\int_0^t\|\nabla \vecu(s)\|^2_{\mL^2}\ds
+\kappa_2\int_0^t\|\vecu(s)\|^2_{\mL^2}\ds
+\kappa_2\int_0^t\mu\|\vecu(s)\|_{\mL^4}^4\ds\nn\\
 &=
 \frac12 \|\vecu_0\|^2_{\mL^2} 
 +
 \frac12\kappa_1 t \sum_{k=1}^\infty \|\vech_k\|^2_{\mL^2}
 +
 \kappa_1\sum_{k=1}^\infty \int_0^t\inpro{\vecu(s)}{\vech_k}_{\mL^2}\,dW_k(s).
\end{align}
By Theorem \ref{theo: main} we have 
\[
\mE\int_0^t \inpro{\vecu(s)}{\vech_k}_{\mL^2}^2\ds
\leq 
\|\vech_k\|_{\mL^2}^2
\mE\int_0^t\|\vecu(s)\|_{\mL^2}^2\ds
<\infty,
\]
hence the process $\to \int_0^t\inpro{\vecu(s)}{\vech_k}_{\mL^2}\,dW_k(s)$ is a martingale on $[0,T]$ In particular 
\[\mE\int_0^t\inpro{\vecu(s)}{\vech_k}_{\mL^2}\,dW_k(s)=0\,,\]
and invoking ~\eqref{eq: BIV2} we obtain 
\begin{align}\label{eq: BIV3}
\frac12\mE   \|\vecu(t)\|^2_{\mL^2} 
&+
\kappa_1\mE\int_0^t\|\nabla \vecu(s)\|^2_{\mL^2}\ds
+\kappa_2\mE\int_0^t\|\vecu(s)\|^2_{\mL^2}\ds
+\kappa_2\mE\int_0^t\mu\|\vecu(s)\|_{\mL^4}^4\ds\nn\\
&=
\frac12 \|\vecu_0\|^2_{\mL^2} 
+
\frac12\kappa_1 t \sum_{k=1}^\infty \|\vech_k\|^2_{\mL^2}
\leq c + ct.
\end{align}
The inequality~\eqref{eq: BIV3} implies
\begin{equation}\label{eq: BIV4}
\mE\int_0^t\|\vecu(s)\|^2_{\mH^1}\ds
\leq c+ct.
\end{equation}
In a similar fashion as in the proof of~\eqref{eq: BIV0}, we obtain the identity 
\begin{align}\label{eq: BIV5}
\frac12 \|\nabla\vecu(t)\|^2_{\mL^2} 
&+\kappa_1\int_0^t\|\Delta \vecu(s)\|^2_{\mL^2}\ds
+\kappa_2\mu\int_0^t\int_D |\vecu(s,\vecx)|^2|\nabla\vecu(s,\vecx)|^2\dvx\ds\nn\\
&+
\kappa_2\int_0^t\|\nabla \vecu(s)\|^2_{\mL^2}\ds
+2\mu\kappa_2\int_0^t \bigl(\inpro{\vecu(s)}{\nabla\vecu(s)}_{\mL^2}\bigr)^2\ds\nn\\
&=
\frac12 \|\nabla\vecu_{0,}\|^2_{\mL^2} 
+
\sum_{k=1}^\infty \int_0^t
R(\vecu(s),\vech_k)\ds\nn\\
&\quad+
\sum_{k=1}^\infty \int_0^t \inpro{\nabla\vecu(s)}{\gamma\vecu(s)\times\nabla\vech_k+\kappa_1\nabla\vech_k}_{\mL^2}\,dW_k(s),
\end{align}
where $R$ is defined as in~\eqref{eq: defR}. We first estimate $R(\vecu,\vech_k)$ by using H\"older inequality as follows
\begin{align*}
R(\vecu,\vech_k)
&:=
 \frac{\gamma}{2}
\inpro{\nabla\vecu}{G_{k}(\vecu)\times\nabla\vech_k}_{\mL^2}
+
\frac12
\inpro{\gamma\vecu\times\nabla\vech_k + \kappa_1\nabla\vech_k}{\nabla G_{k}(\vecu)}_{\mL^2}\nn\\
&\leq 
\frac{\gamma}{2}\|\nabla\vech_k\|_{\mL^\infty}
\bigl(\|\nabla\vecu\|_{\mL^2}\|G_{k}(\vecu)\|_{\mL^2}
+
\|\vecu(s)\|_{\mL^2}\|\nabla G_{k}(\vecu)\|_{\mL^2}
\bigr)\nn\\
&\quad+
\frac{\kappa_1}{2}
\|\nabla\vech_k\|_{\mL^2}\|\nabla G_{k}(\vecu)\|_{\mL^2}\nn\\
&\leq 
\frac{\kappa_1}{4}\|\nabla\vech_k\|_{\mL^2}^2
+
\frac{\gamma}{4}\|\nabla\vech_k\|_{\mL^\infty}\|\vecu\|_{\mH^1}^2
+
\bigl(\frac{\kappa_1}{4}+\frac{\gamma}{4}\|\nabla\vech_k\|_{\mL^\infty}\bigr) \|G_{k}(\vecu)\|_{\mH^1}^2\nn\\
&\leq 
\bigl(\frac{\kappa_1}{4} + \frac{\kappa_1^3}{2}+\frac{\kappa_1^2\gamma}{2}\|\nabla\vech_k\|_{\mL^\infty} \bigr)\|\vech_k\|_{\mH^1}^2\nn\\
&\quad+
\bigl(\frac{\gamma}{4}\|\nabla\vech_k\|_{\mL^\infty}+ \gamma^2\|\vech_k\|_{\mH^1}^2(\frac{\kappa_1}{2}+\frac{\gamma}{2}\|\nabla\vech_k\|_{\mL^\infty})\bigr)
\|\vecu\|_{\mH^1}^2\nn.
\end{align*}
Hence, 
\begin{align}\label{eq: BIV6}
\sum_{k=1}^\infty 
R(\vecu,\vech_k)
&\leq 
c + c\|\vecu\|_{\mH^1}^2.
\end{align}
Again by Theorem \ref{theo: main} we have
\begin{align*}
 \mE \bigl[\sum_{k=1}^\infty 
&\int_0^t\bigl(
\inpro{\nabla\vecu(s)}{\gamma\vecu(s)\times\nabla\vech_k+\kappa_1\nabla\vech_k}_{\mL^2}\bigr)^2\ds\bigr]\nn\\
&\leq
\mE \bigl[\sum_{k=1}^\infty 
\int_0^t\bigl(
\gamma\|\nabla\vech_k\|_{\mL^\infty}
\|\nabla\vecu(s)\|_{\mL^2}\|\vecu(s)\|_{\mL^2}
+
\kappa_1\|\nabla\vecu(s)\|_{\mL^2}\|\nabla\vech_k\|_{\mL^2}
\bigr)^2\ds\bigr]\nn\\
&\leq
2\gamma^2t\bigl(\sum_{k=1}^\infty \|\nabla\vech_k\|_{\mL^\infty}^2\bigr)\mE\|\vecu\|_{L^{\infty}(0,T;\mH^1)}^4
+
2\kappa_1^2\bigl(\sum_{k=1}^\infty \|\nabla\vech_k\|_{\mL^2}^2\bigr)\mE\int_0^t\|\nabla\vecu(s)\|_{\mL^2}^2\ds
<\infty,
\end{align*}
hence the process 
\[t\to \sum_{k=1}^\infty \int_0^t \inpro{\nabla\vecu(s)}{\gamma\vecu(s)\times\nabla\vech_k+\kappa_1\nabla\vech_k}_{\mL^2}\,dW_k(s)\]
is a martingale on $[0,T]$. In particular, 
\[\mE 
\bigl[\sum_{k=1}^\infty 
\int_0^t \inpro{\nabla\vecu(s)}{\gamma\vecu(s)\times\nabla\vech_k+\kappa_1\nabla\vech_k}_{\mL^2}\,dW_k(s)\bigr] = 0\,,\]
and invoking ~\eqref{eq: BIV5}--\eqref{eq: BIV6} and ~\eqref{eq: BIV4} we obtain
\begin{equation*}
\frac12\mE \|\nabla\vecu(t)\|^2_{\mL^2} 
+\kappa_1\mE\int_0^t\|\Delta \vecu(s)\|^2_{\mL^2}\ds
\leq 
\frac12 \|\nabla\vecu_0\|^2_{\mL^2} 
+c+ct,
\end{equation*}
which implies
\[
\mE
\int_0^t\|\Delta \vecu(s)\|^2_{\mL^2}\ds
\leq c+ct\,.
\]
This completes the proof of this lemma.
\end{proof}

For $T>0$, $p\geq 4$ and $\frac12>\beta >\frac14$ let 
\begin{align*}
 \cZ&:= L^2(0,T;\mH^1)\cap L^p(0,T;\mL^4)\cap C\bigl([0,T];X^{-\beta}\bigr)\cap  C([0,T];\mH^1_w)\\
 &=:\td \cZ \cap  C([0,T];\mH^1_w)
\end{align*}
 where $\mH^1_w$ denotes the space $\mH^1$ endowed with the weak topology of $\mH^1$. We will denote by $\cT$ be the supremum of the corresponding four topologies, i.e. the smallest topology on $\cZ$ such that 
the four natural embedding from $\cZ$ are continuous.
\begin{theorem}\label{the: conti_depen}
 Assume that  an $\mH^1$-valued sequence $\{\vecu_{0,l}\}_{l\in\N}$ is convergent weakly in $\mH^1$ to $\vecu_0\in \mH^1$. Let $C_1>0$ be such that 
 $\sup_{l\in\N}\|\vecu_{0,l}\|_{\mH^1}\leq C_1$. Let $(\Omega,\cF,\mF,\mP,\vecu^l,W)$ be a unique solution of~\eqref{eq: sLLB2} with the initial 
 data $\vecu_{0,l}$. Then there exist
 \begin{itemize}
  \item a subsequence $\{l_k\}_k$,
  \item a stochastic basis $(\td\Omega,\td\cF,\td\mF,\td\mP)$,
  \item a standard $\td\mF$-Wiener process
$\td W=(\td W_j)_{j=1}^\infty $ defined on this basis,
\item progressively measurable processes $\td \vecu$, $\{\td \vecu_{l_k}\}_{k\in\N}$ (defined on this basis) with laws supported in $(\cZ,\cT)$
such that 
\begin{align*}
 &\td\vecu_{l_k}\text{ has the same law as }\vecu_{l_k}\text{ on }\cZ\\
 &\td\vecu_{l_k}\goto\td\vecu\text{ in }\cZ\text{ as }k\goto\infty, \,\td\mP-a.s.,\\
 &\text{and } \td\vecu \text{ is a solution of sLLB equation with the initial data } \vecu_0.
\end{align*}
 \end{itemize}
\end{theorem}
\begin{proof}
\underline{Step 1.} From Theorem~\ref{theo: pathwise} and Corollary~\ref{co: strongsolution},  given the inital data $\vecu_{0,l}\in\mH^1$ 
 there exists a unique solution $\vecu^l$ to equation~\eqref{eq: sLLB2} 
defined on the stochastic basis $(\Omega,\cF,\mF,\mP,W)$. Since $C([0,T];\mH^1_w)$ is a non-metric space, we use the Jakubowski's  
version of the Skorokhod theorem proved in \cite{jakubowski}, see also Theorem \ref{theo:Jakubowski} in the Appendix. 
\\
\underline{Step 2.}
We  show that the sequence $\{\vecu^l\}$ of $\cZ$-valued Borel random variables defined on $(\Omega^l,\cF^l,\mF^l,\mP^l)$ satisfies the 
condition of Theorem~\ref{theo:Jakubowski}. 

Let 
\[\mathcal Y(\beta)=W^{\alpha,p}(0,T;X^{-\beta})\cap L^\infty(0,T;\mH^1)\cap L^2(0,T;\mH^2)\]
denote a Banach space endowed with the norm 
\[\|\vecu\|_{\mathcal Y(\beta)}=\|\vecu\|_{W^{\alpha,p}(0,T;X^{-\beta})}+\|\vecu\|_{L^\infty(0,T;\mH^1)}+\|\vecu\|_{L^2(0,T;\mH^2)}\,.\]
By noting that $\{\vecu_{0,l}\}_l$ is uniformly bounded in $\mH^1$ and 
using~\eqref{eq_est1}--\eqref{eq_est2},  
we deduce that  for $\alpha\in[0,\frac12)$, $p\geq 4$ and for all $l=1,2,\cdots$, 
\begin{align*}
 &\vecu^l\in L^{2p}\bigl(\Omega; C([0,T];\mH^1_w)\bigr)\\
 \quad \text{and}\quad
 &\mE^l \|\vecu^l\|_{\mathcal Y(\beta)} \leq c,
\end{align*}
where $c$ is a positive constant  only depending on $C_1$, $p\geq  1$ and $h$.
Let 
\[
\cB_R(\beta):=\{\vecv\in\mathcal Y(\beta)
: \|\vecv\|_{\mathcal Y}\leq R\}. 
\]
By the Chebyshev inequality and the above uniform bound of $\{\vecu^l\}$, we infer that 
\begin{equation}\label{eq: cd1}
 \sup_{l\in\N}\mP^l\bigl(\{\vecu^l\in\cB_R(\beta)\}\bigr)\geq 1-\frac{c}{R^2}
\end{equation}
The following compact embedding
\begin{equation}\label{eq: com_em}
 W^{\alpha,p}(0,T;X^{-\beta_1})\cap L^p(0,T;\mH^1)\cap L^2(0,T;\mH^2)
 \hookrightarrow \td\cZ,
\end{equation}
holds for $\beta_1\in(0,\beta)$. Therefore, 
\begin{equation}\label{eq: cd2}
 \cB_R\left(\beta_1\right)\text{  is a compact subset in }\td\cZ. 
\end{equation}
By Theorem 2.1 in~\cite{Strauss1966} we have for any $\beta\ge 0$ a continuous imbedding\footnote{In fact, the continuous imbedding is not explicitly stated in Theorem 2.1 but in our case it can be easily deduced from the proof.}.
\[
  L^\infty(0,T;\mH^1)\cap C([0,T];X^{-\beta}_w)\subset C([0,T];\mH^1_w)\,
\]
 As a consequence we find that for a certain $r>0$ we have $\cB_R\left(\beta_1\right)\subset \td\cZ\cap C([0,T];\mB^1_w(r))$ where $C([0,T];\mB^1_w(r))$ is the metric subspace in $C([0,T];\mH^1_w)$ and $\mB^1_w(r)$ was defined on p.4. 
Let $\{\vecv_n\}$ be a sequence in $\cB_R\left(\beta_1\right)$. Then $\{\vecv_n\}$ is uniformly bounded in 
$W^{\alpha,p}(0,T;X^{-\beta})\cap L^\infty(0,T;\mH^1)\cap L^2(0,T;\mH^2)$. It follows from~\eqref{eq: com_em}--\eqref{eq: cd2} that 
there exist a subsequence of $\{\vecv_n\}$ (still denoted by $\{\vecv_n\}$) and $\vecv\in \cB_R\left(\beta_1\right)$ satisfying
\[
 \vecv_n\goto \vecv\quad \text{ in }\td\cZ,\quad\text{and}\quad  \vecv_n\goto \vecv\quad \text{weak$^\star$\,\,in\,\,}L^\infty(0,T;\mH^1),
\]
which implies
\[
 \lim_{n\goto\infty}\sup_{s\in[0,T]}|\inpro{\vecv_n(s)-\vecv(s)}{\vech}_{\mH^1}|=0\quad\forall \,\vech\in\mH^1.
\]
Therefore, $\vecv_n\goto \vecv$ in $\td\cZ\cap C([0,T];\mB^1_w(r))$. This together with~\eqref{eq: cd2} implies 
\begin{equation}\label{eq: cd3}
 \cB_R\left(\beta_1\right)\text{  is a compact subset in }\cZ. 
\end{equation}

Now, taking into account ~\eqref{eq: cd1},~\eqref{eq: cd3}, the proof of the theorem follows from Theorem~\ref{theo:Jakubowski}.

\end{proof}

\noindent
Let us recall that by Corollary \ref{co: strongsolution} equation \eqref{eq: mainu} has a unique weak solution that, in view of Theorem \ref{theo: main}, defines a $\mH^1$-valued Markov process $\vecu$. Therefore, we can define its transition semigroup: for any $\phi\in B_b(\mH^1)$, i.e. a bounded and Borel function $\phi:\mH^1\to\R$ we define 
\begin{equation}\label{eq: defPt}
 P_t\phi(\vecu_0) := \mE[\phi(\vecu(t;\vecu_0))], \quad\forall \vecu_0\in\mH^1,
\end{equation}
 where $\vecu(t,\vecu_0)$ stands for the process $\vecu$ starting at time $t=0$ at $\vecu(0)=\vecu_0$.
\noindent
The next result states the sequentially weak Feller property of $P_t$. 
\begin{lemma}\label{lem: Feller}
Let $\phi:\mH^1\goto \R$ be a bounded and sequentially weakly continuous function and let 
 $\vecu_{0,l}\goto\vecu_0$ weakly in $\mH^1$ as $l\goto\infty$. Then for every $t\ge 0$ 
\[
P_t\phi(\vecu_{0,l})\goto P_t\phi(\vecu_0)\quad
\text{as }l\goto\infty\,.
\]
\end{lemma}
\begin{proof}
Assume that $\vecu_{0,l}\goto\vecu_0$ weakly in $\mH^1$ as $l\goto\infty$. 
By Theorem~\ref{the: conti_depen}, there exist a subsequence of $\vecu_{l}$ (still denoted by $\vecu_{l}$), 
a stochastic basis $(\td\Omega,\td\cF,\td\mF,\td\mP)$,
an $\R^\infty $-valued standard $\td\mF$-Wiener process
$\td W=(\td W_j)_{j=1}^\infty $ defined on this basis, progressively measurable processes $\td \vecu$ and $\{\td \vecu_{l}\}_{l\in\N}$ (defined on this basis) 
with laws supported in $(\cZ,\cT)$ such that 
\begin{equation}\label{eq_new1}
 \td\vecu_{l}\text{ has the same law as }\vecu_{l}\text{ on }\cZ
 \end{equation}
 and 
\[\td\vecu_{l}\goto\td\vecu\text{ in }\cZ\text{ as }l\goto\infty,\quad\td\mP-a.s.\]
Hence, 
\begin{equation}\label{eq: Feller1}
 \td\mE[\phi(\td\vecu(t))] = \mE[\phi(\vecu(t;\vecu_0))] =: P_t\phi(\vecu_0),
\end{equation}
and $\td\vecu^l\goto\td\vecu$  in $C([0,T;\mH^1_w)$, $\td\mP$-a.s.
This together with the sequential weak continuity of $\phi$ implies  
\[
\phi(\td\vecu_l(t))\goto  \phi(\td\vecu(t))\quad\text{in }\R.
\]
Therefore, since the function $\phi$ is bounded, by the Lebesgue Dominated Convergence Theorem we infer that 
\begin{equation}\label{eq: Feller2}
 \lim_{l\goto\infty}\td\mE[\phi(\td\vecu_l(t))] = \td\mE[\phi(\td\vecu(t))].
\end{equation}
Note that equality of laws \eqref{eq_new1} yields equality of laws of $\td\vecu_l(t)$ and $\vecu_l(t)$ for every $t\ge 0$. Thus by~\eqref{eq: Feller1}--\eqref{eq: Feller2} we obtain
\[
\lim_{l\goto\infty} P_t\phi(\vecu_{0,l}) = \lim_{l\goto\infty} \mE[\phi(\vecu_l)]
=
\lim_{l\goto\infty} \td\mE[\phi(\td\vecu_l(t))] = \td\mE[\phi(\td\vecu(t))]=P_t\phi(\vecu_0),
\]
and the lemma follows.
\end{proof}
\begin{theorem}\label{theo: invariantmeasure}
Let $D\subset \R$ or $D\subset \R^2$. Then there exists at least one invariant measure for equation~\eqref{eq: sLLB2}.
\end{theorem}
\begin{proof}
Lemma~\eqref{lem: Feller} implies that the semigroup $\{P_t\}_{t\geq 0}$ is sequentially weakly Feller in $\mH^1$. 
Using the Chebyshev inequality and  Lemma~\ref{lem: boundInProba}, we infer that for every $T>0$ and $R>0$
\begin{align*}
\frac{1}{T}\int_0^{T}\mP(\{\|\vecu(s;\vecu_0)\|_{\mH^1}>R\})\ds
\leq 
\frac{1}{TR^2}\int_0^{T}\mE[\|\vecu(s;\vecu_0)\|_{\mH^1}^2]\ds
\leq 
\frac{c+cT}{TR^2},
\end{align*}
where $c$ is the constant only depending on $\vecu_0$ and $h$.
Hence, thanks to the Maslowski-Seidler theorem, see \cite{Maslowski1999} or Theorem~\ref{theo: theorem11.7},
we infer that there exists at least one invariant measure for equation~\eqref{eq: sLLB2}.
\end{proof}

\section{Appendix}\label{sec: appe}
\begin{lemma}\label{lem: appendix1}
Assume that $E$ is a separable Hilbert space, $p\in[2,\infty)$ and $\alpha\in (0,\frac12)$. 
Then there exists a constant c depending on $T$ and $\alpha$ such that for any progressively measurable 
process $\xi = (\xi_j)_{j=1}^{\infty}$ there holds
\begin{equation*}
\mE\| \sum_{j=1}^{\infty}I(\xi_j)\|^p_{W^{\alpha,p}(0,T;E)}
\leq 
c\mE\int_0^T\bigl(\sum_{j=1}^{\infty}|\xi_j(t)|_E^2\bigr)^{\frac{p}{2}}\dt,
\end{equation*}
where $I(\xi_j)$ is defined by 
\begin{equation*}
I(\xi_j) := \int_0^t\xi_j(s)\,dW_j(s),\quad t\geq 0.
\end{equation*}
In particular, $\mP$--a.s. the trajectories of the process $I(\xi_j)$ belong to $W^{\alpha,2}(0,T;E)$.
\end{lemma}
\begin{lemma}\label{lem: appendix2}\cite[Corollary 19]{Simon1990}
Suppose $s\geq r$, $p\leq q$ and $s-1/p\geq r-1/q$ 
($0<r\leq s<1$, $1\leq p\leq q\leq \infty$). Let $E$ be a 
Banach space and $I$ be an interval of $\R$. Then
\[
W^{s,p}(I;E)\hookrightarrow W^{r,q}(I;E).
\]
\end{lemma}
\noindent
Let us recall the Maslowski-Seidler theorem \cite{Maslowski1999} about the existence of an invariant measure.
\begin{theorem}\label{theo: theorem11.7}
Assume that 
\begin{enumerate}
 \item the semigroup $\{P_t\}_{t\geq 0}$ is sequentially weakly Feller  in $\mH^1$;
 \item 
  there exists $T_0\geq 0$ such that for any $\epsilon>0$ 
  there exists $R>0$ satisfying
\[
\sup_{T> T_0}\frac{1}{T}\int_0^{T}\mP(\{\|\vecu(s;\vecu_0)\|_{\mH^1}>R\})\ds \leq \epsilon,
\]
\end{enumerate}
Then there exists at least one invariant measure for equation~\eqref{eq: sLLB2}.
\end{theorem}
\noindent
Let us recall the Jakubowski's version of the Skorokhod Theorem \cite{jakubowski}
\begin{theorem}\label{theo:Jakubowski}
Let $(\cX,\tau)$ be  a topological space such that there exists a sequence $\{f_m\}$ of 
continuous functions $f_m:\cX\goto\R$ that separates points of $\cX$. 
Let $\{X_n\}$ be a sequence of $\cX$-valued Borel random variables defined on $(\Omega^n,\cF^n,\mP^n)$. Suppose that for evey $\epsilon>0$ 
there exists a compact subset $K_\epsilon\subset\cX$ such that 
\[
 \sup_{n\in\N}\mP^n(\{X_n\in K_{\epsilon}\})> 1-\epsilon.
\]
Then there exist a subsequence $\{n_k\}_{k\in\N}$, a sequence $\{Y_k\}_{k\in\N}$ of $\cX$--valued Borel random variables 
and an $\cX$--valued Borel random variable $Y$ defined on a certain probability space $(\Omega,\cF,\mP)$ such that
\[
 \cL(X_{n_k}) = \cL(Y_k),\quad k=1,2,\cdots 
\]
and
\[
 Y_k\goto^{\tau}Y\quad \text{as } k\goto\infty,\quad \mP-a.s.
\]

\end{theorem}

\bibliographystyle{myabbrv}
\bibliography{mybib}
\end{document}